\documentclass{amsart}
\title[The families SW invariants of smooth families of K\"ahler surfaces]{The families Seiberg-Witten invariants of smooth families of K\"ahler surfaces}
\author{Joshua Celeste}
\date{\today}

\usepackage{graphicx}
\usepackage[style=alphabetic,citestyle=alphabetic,url=false,isbn=false,doi=false,url=false,backend=bibtex,maxnames=5]{biblatex}
\renewbibmacro{in:}{}

\edef\restoreparindent{\parindent=\the\parindent\relax}

\usepackage{braket}
\usepackage{tensor}
\usepackage{xcolor}
\usepackage{amsmath, amssymb, amsthm, amsfonts}
\usepackage{bbm}
\usepackage{mathtools}
\usepackage{epstopdf}
\usepackage{enumerate}
\usepackage{color}
\usepackage{tikz}
\usepackage{mathrsfs}
\usepackage{tikz-cd}
\usepackage{hyperref}
\usepackage{pdfpages}
\usepackage{cleveref}
\usepackage{spectralsequences}

\newcommand{\Z}{\mathbb{Z}}

\newcommand{\R}{\mathbb{R}}
\newcommand{\C}{\mathbb{C}}
\renewcommand{\P}{\mathbb{P}}

\newcommand{\partialbar}{   
	\overline{\partial}
}

\newcommand\SW{\text{SW}}

\newcommand \FSW{\text{FSW}}

\theoremstyle{plain}
\newtheorem{theorem}{Theorem}[section]
\newtheorem{lemma}[theorem]{Lemma}
\newtheorem{corollary}[theorem]{Corollary}
\newtheorem{proposition}[theorem]{Proposition}

\theoremstyle{definition}

\newtheorem{remark}[theorem]{Remark}
\newtheorem{example}[theorem]{Example}

\theoremstyle{remark}

\begin{document}
	\begin{abstract}
		We consider a generalisation of the Seiberg-Witten invariant to the families Seiberg-Witten invariants of a smooth family of 4-manifolds with fibres diffeomorphic to a 4-manifold \(X\). Of particular interest is the special case when the family has a smoothly varying K\"ahler structure. We obtain a general computation of the invariants when \(b_1=0\) in terms of characteristic classes of some vector bundles corresponding to the cohomology groups of holomorphic line bundles of the family. Finally, we apply the formula to some examples of K\"ahler families where some more further explicit computations can be made. We consider a family of \(\C\P^2\)'s obtained from the projectivisation of a rank 3 complex vector bundle, a family of \(\C\P^1\times \C\P^1\)'s obtained as the fibre product of the projectivisation of two rank 2 complex vector bundles and a family with fibres being the blowup of a K\"ahler surface at a point known as the universal blowup family.
	\end{abstract}
	\maketitle
	\section{Introduction}
	\label{sec:1}
	The Seiberg-Witten invariant has proven useful in studying smooth 4-manifolds, particularly in regards to compact K\"ahler surfaces where the invariant may be computed. On a K\"ahler surface, there is a canonical \(\text{spin}^c\) structure, with any other \(\text{spin}^c\) structure obtained via a choice of Hermitian line bundle \(L\). One may then compute the Seiberg-Witten invariant when \(b_1=0\) in terms of the dimensions of the cohomology groups of \(L\).
	
	The Seiberg-Witten invariant can be extended for application to smooth families of 4-manifolds, this idea was first suggested by Donaldson in \cite{donaldson_seiberg-witten_1996} and applied to some specific families in \cite{ruberman_obstruction_1998}, \cite{ruberman_positive_2002}, \cite{nishinou_nontrivial_2002}. The idea of a families Seiberg-Witten invariant was then pursued in further generality by \cite{li_family_2001}. We shall concern ourselves with an even more general construction as seen in \cite{baraglia_obstructions_2019} and \cite{baraglia_bauerfuruta_2022}, where for each \(n\ge 0\), one may define the families Seiberg-Witten invariants which are valued in the cohomology ring of the base of the family.
	
	Consider a smooth family of 4-manifolds \(X\hookrightarrow E\rightarrow B\) with base \(B\) and fibres \(X_b\) diffeomorphic to some \(4\)-manifold \(X\). If the family admits a smoothly varying K\"ahler structure, one can obtain a canonical \(\text{spin}^c\) structure on the vertical tangent bundle. Any other \(\text{spin}^c\) structure for the family is then obtained via tensoring by a Hermitian line bundle. Similar techniques to the unparametrised case can be applied to obtain the main result of this paper, namely a computation of the families Seiberg-Witten invariants for K\"ahler families when \(b_1(X)=0\) and the dimensions of the cohomology groups \(H^i(X_b,L_b)\) are constant for \(b\in B\). That is, we have the following theorem.
	
	\begin{theorem}
		\label{thm:mainthm}
		Let \(X\hookrightarrow E\rightarrow X\) be a smooth family of compact K\"ahler surfaces with a smoothly varying K\"ahler structure and \(b_1(X)=0\) such that the following assumptions hold
		\begin{enumerate}
			\item For all \(b\in B\), the line bundle \(L_b\) has first Chern class \(c_1(L_b)\) which is represented by a \((1,1)\) form. \\
			\item The dimensions \(h^i\) of \(H^i(X_b,L_b)\) for \(i=0,1,2\) are independent of \(b\).
		\end{enumerate}
		Define \(\delta=m+n-h^0+1\), the families Seiberg-Witten invariants in the K\"ahler chamber are given by
		\begin{equation}
			\label{eqn:maineqn}
			\FSW_n(E,L)=\sum_{m=0}^{h^1-h^2+\rho_g}c_{h^1-h^2+\rho_g-m}(H^{2,0})\Gamma_{m,n}
		\end{equation}
		where \(\Gamma_{m,n}\in H^{2\delta}(B;\Z)\) is given by
		\begin{align}
			\Gamma_{m,n}&=(-1)^n\sum_{i=\max\{\delta-m,0\}}^{\delta}\sum_{j=0}^{\min\{i,m-\delta+i\}}s_j(V^2)c_{\delta-i}(V^1)s_{i-j}(V^0){h^1-h^2-\delta+i-j\choose m-\delta+i-j}.
		\end{align}
		for \(1\le h^0\le m+n+1\) and is zero otherwise.
		
		Furthermore, the \(\Gamma_{m,n}\) satisfy the following recursion relation
		\[\Gamma_{m,n+h^0} - c_1(V^0)\Gamma_{m,n+h^0-1} + c_2(V^0)\Gamma_{m,n+h^0-2}+\dots = 0.\]
	\end{theorem}
	
	The structure of the paper is as follows. We begin in \Cref{sec:2} with a brief review of the computation of the Seiberg-Witten invariant for K\"ahler surfaces in the unparametrised case via the obstruction bundle technique. This short overview is instructive since many of the methods apply more generally to the families case. \Cref{sec:3} contains an overview of Seiberg-Witten theory in the families setting. We discuss the construction of the invariants when \(b_1(X)=0\). \Cref{sec:computation for general class of kahler families} then consists of the proof of the main theorem of the paper, \Cref{thm:mainthm}, the computation of the families Seiberg-Witten invariants for families of K\"ahler surfaces. The following sections then contain applications of the computation of the families Seiberg-Witten invariants for K\"ahler families to a series of examples. In \Cref{sec:examples 1} we consider two example families which have a similar means of obtaining an expression for the invariant. These families are a family with fibres diffeomorphic to \(\C\P^2\) obtained from the projectivisation of a rank 3 complex vector bundle, and a family with fibres diffeomorphic to \(\C\P^1\times\C\P^1\) obtained from the fibre product of the projectivisations of two rank 2 complex vector bundles. Finally, we conclude with a computation of the families Seiberg-Witten invariants for the universal blowup family. This family has base being a K\"ahler surface \(X\) and its fibre at a point \(x\) is the blowup of \(X\) at \(x\).
	\section{Review of the computation of the Seiberg-Witten invariant on K\"ahler surfaces}
	\label{sec:2}
	
	 In this section we briefly review the computation of the Seiberg-Witten invariant on K\"ahler surfaces with \(b_1=0\), for a more in depth treatment see \cite{salamon_spin_1999}. Recall that generally for a smooth compact oriented 4-manifold \(X\), given a choice of metric and \(\text{spin}^c\) structure \(\mathfrak{s}\) with associated line bundle \(\mathcal{L}_{\mathfrak{s}}\), the Seiberg-Witten moduli space is defined as the solution set of the Seiberg-Witten equations modulo the action of the gauge group \(\mathcal{G}=\text{Maps}(X,S^1)\). We choose a generic perturbation \(\eta\) so that the Seiberg-Witten equations in \((A,\Phi)\) do not admit reducible solutions with \(\Phi=0\) and the moduli space is a smooth compact oriented manifold of dimension
	 \[d:=d(X,\mathfrak{s})=\frac{\braket{c(L_{\mathfrak{s}})^2,[X]}-2\chi(X)-3\sigma(X)}{4}.\]
	 where \(\chi(X)\) is the Euler characteristic of \(X\) and \(\sigma(X)\) is its signature, defined as \(\sigma(X):=b^+-b^-\). The space of perturbations \(\mathcal{W}\) for which the Seiberg-Witten equations necessarily admits reducble solutions is called the wall  and is an affine subspace of codimension \(b^+\) in the space of perturbations. Consequently if \(b^+=1\), the complement of the wall is split into two connected components called chambers, and it is connected otherwise. We denote a choice of chamber by \(\phi\) when necessary, note that any two generic perturbations in a choice of chamber are path connected and there is a cobordism between the corresponding moduli spaces.
	 
	 If \(b_1(X)=0\), we may define the reduced gauge group 
	 \[\mathcal{G}_{0}:=\{g\in \mathcal{G}_b:g=e^{if}, f:X\longrightarrow \R,\int_{X}f\text{vol}_{X}=0\}\]
	 which fits into an exact sequence
	 \[0\rightarrow  \mathcal{G}_0\rightarrow \mathcal{G}\rightarrow S^1\rightarrow 0.\]
	 Define \(\widetilde{\mathcal{M}}\) to be the space of solutions to the Seiberg-Witten equations modulo the action of \(\mathcal{G}_0\), then there is a principle circle bundle \(\widetilde{\mathcal{M}}\rightarrow\mathcal{M}\) with associated line bundle \(\mathcal{L}\rightarrow\mathcal{M}\) defined by the action of the circle by multiplication. Let \(\tau=c_1(\mathcal{L})\in H^2(\mathcal{M},\Z)\), the Seiberg-Witten invariant is defined by
	 \[SW(X,\mathfrak{s}_L,\phi):=
	 \begin{cases}
	 	\int_{\mathcal{M}}\tau^{\frac{d}{2}} & \text{if }d\text{ is even} \\
	 	0 & \text{otherwise.}
	 \end{cases}
 	\]
 	Since there is a cobordism between any two moduli spaces, given a generic choice of metric and perturbation in a chamber, the Seiberg-Witten invariant of \(X\) only depends on the choice of \(\text{spin}^c\) structure when \(b^+>1\) and also a choice of chamber when \(b^+=1\).
	 
	 Now assume that \(X\) is a K\"ahler surface, the symplectic and almost-complex structures give rise to a canonical \(\text{spin}^c\) structure \(\mathfrak{s}_{\text{can}}\) with spinor bundle \(W_{\text{can}}\rightarrow X\). The Levi-Civita connection induces a canonical \(\text{spin}^c\) connection \(\nabla_{\text{can}}\) on \(W_{\text{can}}\), and any other \(\text{spin}^c\) structure is obtained via twisting by a Hermitian line bundle \(L\rightarrow X\). Let \(\mathcal{A}(L)\) denote the space of Hermitian connections on \(L\) and \(\mathcal{A}(\mathfrak{s}_L)\) denote the space of connections on \(\mathcal{L}_{\mathfrak{s}_L}=\text{det}(W_{\text{can}}\otimes L)=K_X^*\otimes L^2\), the latter is in correspondence with \(\text{spin}^c\) connections on \(W_{\text{can}}\otimes L\) that are compatible with the Levi-Civita connection. Given \(2A\in\mathcal{A}(L)\), there is a unique \(\text{spin}^c\) connection compatible with the Levi-Civita connection for the twisted \(\text{spin}^c\) structure \(\mathfrak{s}_L=\mathfrak{s}_{\text{can}}\otimes L\) and the induced connection on \(\mathcal{A}(\mathfrak{s}_L)\) is denoted \(A_{\text{can}}+2A\in\mathcal{A}(\mathfrak{s}_L)\). Note that the factor of 2 in front of \(A\) is simply to reduce unnecessary factors in the Seiberg-Witten equations. Since a \(\text{spin}^c\) structure on a K\"ahler surface \(X\) only depends on a choice of line bundle \(L\), we often write the Seiberg-Witten invariant of \(X\) corresponding to the \(\text{spin}^c\) structure \(\mathfrak{s}_L\) as \(\SW(X,L)\) when \(X\) is K\"ahler. The Seiberg-Witten equations in the K\"ahler case then simplify to the following form.
	
	\begin{proposition}
		\label{prop:kahler sw equations}
		Let \(X\) be a K\"ahler surface and \(\eta\in i\Omega^{2,+}(X)\), the perturbed Seiberg-Witten equations then acts on pairs \((A_{\text{can}}+2A,\Phi)\) and are given by
		\begin{align}
			\overline{\partial}_A\varphi_0+\overline{\partial}^*_A\varphi_2&=0 \label{kahler sw eqn 1} \\
			2(F_A+\eta)^{0,2}&=\overline{\varphi_0}\varphi_2 \label{kahler sw eqn 2}\\
			4i(F_{A_{\text{can}}}+F_A+\eta)_{\omega}&=|\varphi_2|^2-|\varphi_0|^2 \label{kahler sw eqn 3}
		\end{align}
		where \(\Phi=(\varphi_0,\varphi_2)\in\Omega^{0,0}(X,L)\times\Omega^{0,2}(X,L)\), \(\eta\) is a self-dual imaginary valued 2-form and for a 2-form \(\tau\in\Omega^2(X,\C)\), \(\tau_{\omega}:X\longrightarrow\C\) is defined by
		\[\omega\wedge\tau:=\tau_{\omega}\omega\wedge\omega.\]
	\end{proposition}
	
	
	\begin{remark}
		Note that \(\varphi_0\) is valued in the line bundle \(L\) and there is no naturally defined complex conjugation. Hence \(\overline{\varphi}\) should be interpreted as a section of the bundle \(\overline{L}=L^*\) with the reversed complex structure and the product \(\overline{\varphi}_0\varphi_2\) should be interpreted as the tensor product. 
		
	\end{remark}

	We now present a theorem of Bradlow in \cite{bradlow_vortices_1990}, which is key to obtaining a more general computation of the Seiberg-Witten invariant for K\"ahler surfaces.
	
	\begin{theorem}[Bradlow]
		\label{thm:bradlow's theorem}
		Let \((X,\omega,J,g)\) be a K\"ahler surface and \(L\longrightarrow X\) a Hermitian line bundle, if
		\[0\le c_1(L)\cdot[\omega]<\frac{c_1(K^*)\cdot[\omega]}{2}+\lambda\text{Vol}(X),\]
		then there is a natural bijection
		\[\mathcal{M}(X,\Gamma_L,g,i\pi\lambda\omega)\cong\text{Div}^{\text{eff}}(X,c_1(L))\]
		and if
		\[\frac{c_1(K^*)\cdot[\omega]}{2}+\lambda\text{Vol}(X)<c_1(L)\cdot[\omega]\le c_1(K^*)\cdot[\omega],\]
		then there is a natural bijection
		\[\mathcal{M}(X,\Gamma_L,g,i\pi\lambda\omega)\cong \text{Div}^{\text{eff}}(X,c_1(K^*)-c_1(L)).\]
	\end{theorem}
	It follows from the above theorem and its proof that by choosing a perturbation \(\eta=i\pi\lambda\omega\) with \(\lambda>0\) sufficiently large, the moduli space for this particular perturbation is empty unless \(L\) is a holomorphic line bundle. If \(L\) is indeed a holomorphic line bundle, when \(b_1(X)=0\) the holomorphic structure is unique. Hence the points of the moduli space can then be identified with the non-zero holomorphic sections of \(L\) up to gauge equivalence of line bundles.
	\[\mathcal{M}(X,\mathfrak{s}_L,g,i\pi\lambda\omega)\cong\frac{H^0(X,L))\setminus\{0\}}{\C^*}=\P(H^0(X,L)).\]
	This choice of perturbation defines a chamber known as the \textit{K\"ahler chamber}, when \(b^+=1\) we denote the invariant in this chamber by \(SW^+(X,L)\). The observations above yield the following computation of the Seiberg-Witten invariant for a K\"ahler surface \(X\) with \(b_1=0\), in terms of the dimensions \(h^i\) of the cohomology groups of \(H^i(X,L)\) and the geometric genus of \(X\), defined by \(\rho_g:=(b^+-1)/2\).
	
	\begin{theorem}
		\label{thm:general computation of the SW invariant on Kahler surfaces}
		Let \(X\) be a K\"ahler surface with \(b_1(X)=0\) and \(L\longrightarrow X\) a holomorphic line bundle, if \(p_g>0\) and \(h^0>0\) then
		\[\SW(X,L)={h^1-h^2\choose h^1-h^2+\rho_g}.\]
		
		If \(\rho_g=0,\) and \(\chi(X,L)\ge 1\) then
		\[\SW^+(X,L)=1\] 
		and is zero otherwise.
	\end{theorem}
	
	Note that when \(\rho_g>0\) and \(\chi(X,L)=\rho_g+1\), then the formula for the Seiberg-Witten invariant in \Cref{thm:general computation of the SW invariant on Kahler surfaces} can be rewritten as
	\[\text{SW}(X,L)=(-1)^{h^0-1}{p_g-1\choose h^0-1},\qquad \text{ if } h^1-h^2<0<h^0\]
	furthermore the Seiberg-Witten invariant is only non-zero when \(h^1-h^2<0<h^0\).
	
	\begin{proof}[Proof of Theorem 2.4]
		We shall use the previous theorem of Bradlow for the computation of the moduli space by choosing a perturbation \(\eta=i\pi\lambda\omega\) for \(\lambda\) sufficiently large. 	Since \(b_1(X)=0\), the line bundle \(L\) admits at most one holomorphic structure and hence
		\[\mathcal{M}(X,\mathfrak{s}_L,g,i\pi\lambda\omega)=\mathcal{M}\cong\frac{H^0(X,L))\setminus\{0\}}{\C^*}=\P(H^0(X,L)).\]
		
Although this choice of perturbation is not necessarily regular, a computation of the invariant can still be made. The moduli space can be obtained as the zero set of a Fredholm section of a Hilbert bundle, defined by the Seiberg-Witten equations in conjunction with appropriate gauge fixing conditions. It is useful to define \(f\) with the Coulomb gauge fixing condition applied, so that \(\widetilde{\mathcal{M}}=f^{-1}(0)\). The map \(f\) is \(S^1\) equivariant and the full moduli space is obtained as the zero set of the induced map \(\overline{f}\) which acts on the quotient.
		\Cref{thm:bradlow's theorem} asserts that the corresponding moduli space is still a smooth compact manifold. Furthermore if \(T_x\mathcal{M}=\text{ker}(d\overline{f}_{(A,\varphi)})\), then the cokernels of \(d\overline{f}\) are of constant dimension and form a vector bundle called the obstruction bundle, and a generic perturbation of the moduli space is isotopic to a section of the obstruction bundle. Hence the Seiberg-Witten invariant in the K\"ahler chamber can be computed in terms of the non-generic perturbation by cupping the relevant cohomology class with the Euler class of the obstruction bundle, that is
		\[\SW(X,\mathfrak{s}_L)=\int_{\P(H^0(X,L))}e(\text{Obs})\tau^{d/2}\]
		where \(d\) is the expected dimension of the moduli space. Moreover, since K\"ahler surfaces are of simple-type, when \(\rho_g>0\), the invariant is zero unless the expected dimension of the moduli space is zero, in such a case the invariant is simply given by integrating the Euler class of the obstruction bundle. Since the obstruction bundle is a complex vector bundle when \(X\) is K\"ahler, the Euler class can equivalently be computed as the top degree Chern class of the obstruction bundle. From the linearisation of the Seiberg-Witten equations with the Coulomb gauge fixing condition applied, one may define a corresponding operator by \(\mathcal{D}\), similarly by applying the action of the full gauge group, an operator \(\widetilde{\mathcal{D}}\), the latter of these fits into a long exact sequence.
		\begin{lemma}
			\label{lem:mrowkas exact sequence}
			{(Mrowka)} Let \(2A\in\mathcal{A}(L)\) and \(\varphi_0\in C^{\infty}(X,L)\) be nonzero with \(F_A^{0,2}\) and \(\partialbar_A\varphi_0=0\). Then there is an exact sequence as follows.
			
			\[\begin{tikzcd}
				0 & {H^0(X,\mathcal{O})} & {H^0(X,\mathcal{E}_A)} & {\text{ker}(\mathcal{\widetilde{D}}_{A,\varphi})} \\
				& {H^1(X,\mathcal{O})} & {H^1(X,\mathcal{E}_A)} & {\text{coker}(\mathcal{\widetilde{D}}_{A,\varphi})} \\
				& {H^2(X,\mathcal{O})} & {H^2(X,\mathcal{E}_A)} & 0
				\arrow[from=1-1, to=1-2]
				\arrow[from=1-2, to=1-3]
				\arrow[from=1-3, to=1-4]
				\arrow[from=1-4, to=2-2]
				\arrow[from=2-2, to=2-3]
				\arrow[from=2-3, to=2-4]
				\arrow[from=2-4, to=3-2]
				\arrow[from=3-2, to=3-3]
				\arrow[from=3-3, to=3-4]
			\end{tikzcd}\]
			where \(\mathcal{O}\) is the structure sheaf of holomorphic functions on \(X\), \(\mathcal{E}_A\) is the sheaf of holomorphic sections of \(L\) with holomorphic structure given by \(\partialbar_A\).
		\end{lemma}
		\begin{proof}
			\cite{salamon_spin_1999}[Lemma 12.6]
		\end{proof}
		We further have that
		\[\text{ker}(df_{(A,\Phi)})\cong\text{ker}(\mathcal{D}_{A,\Phi}),\qquad\text{coker}(df_{(A,\Phi)})\cong\text{coker}(\mathcal{D}_{A,\Phi})\] and so
		\begin{align*}
			\text{ker}(d\overline{f}_{[A,\Phi]})&\cong\text{ker}(\tilde{\mathcal{D}}_{A,\Phi})\cong\text{ker}(\mathcal{D}_{A,\Phi})/R(0,i\Phi) \\
			\qquad\text{coker}(d\overline{f}_{[A,\Phi]})&\cong\text{coker}(\tilde{\mathcal{D}}_{A,\Phi})\cong\text{coker}(\mathcal{D}_{A,\Phi}).
		\end{align*}
		The exact sequence of \Cref{lem:mrowkas exact sequence} induces an exact sequence of vector bundles over \(\widetilde{\mathcal{M}}=f^{-1}(0)\), by applying the remaining \(S^1\) action, one obtains that the cokernels have constant complex dimension of \(h^1-h^2+\rho_g\) and \(\text{dim}_{\C}(\text{ker}(\mathcal{D}_{A,\Phi}))=h^0-1\), from which it follows that \(T_x\mathcal{M}=\text{ker}(d\overline{f}_{(A,\varphi)})\). Applying the \(S^1\) action gives an exact sequence of vector bundles over the moduli space \(\mathcal{M}\)
		\[0\rightarrow H^{h^1}\rightarrow \text{Obs}\rightarrow \C^{\rho_g}\rightarrow H^{h^2}\rightarrow 0\]
		where \(H=\mathcal{O}(1)\) is the hyperplane line bundle lying over \(\mathcal{M}\), by letting \(x=c_1(H)\), it follows that \(e(\text{Obs})\) is given by the \(2(\rho_g+h^1-h^2)\)-th degree term of \((1+x)^{h^1-h^2}\). The line bundle \(\mathcal{L}\rightarrow\mathcal{M}\) defining \(\tau\) is also easily seen to be the tautological line bundle \(\mathcal{O}(-1)\), consequently, the Seiberg-Witten invariant can be written as a multiple of an appropriate power of the integral of \(x\), yielding the desired formulae.
	\end{proof}
	\section{The families Seiberg-Witten invariants}
	\label{sec:3}
		We now briefly provide an overview of families Seiberg-Witten theory and the families invariants before the main computation of the following section. Let \(X\) be a compact, oriented, smooth 4-manifold with \(b^+(X)>0\) and \(B\) a compact smooth manifold. We say that a \textit{smooth family of 4-manifolds} with fibres \(X\) and base \(B\), is a smooth locally trivial fibre bundle \(\pi:E\longrightarrow B\) which is fibrewise oriented, such that for each \(b\) there exists an orientation preserving diffeomorphism \(\pi^{-1}(b)\cong X\). Since the fibres are diffeomorphic to \(X\), the fibre over \(b\in B\) shall be denoted \(X_b\). Fix some smooth family of 4-manifolds \(X\hookrightarrow E\rightarrow B\) and further assume that \(b_1(X)=0\). Choose a smoothly varying family of fibrewise metrics \(\{g_b\}\), this turns the vertical tangent bundle \(T(E/B)\) into a Riemannian vector bundle over \(E\). Further, choose a smoothly varying family of self-dual perturbations \(\{\eta_b\}\) to the Seiberg-Witten equations and a \(\text{spin}^c\) structure \(\mathfrak{s}\) on the vertical tangent bundle with respect to the family of metrics \(\{g_b\}\). Note that any two such \(\text{spin}^c\) structures are related by tensoring via a Hermitian line bundle over \(E\). Define the \textit{parameter space} of metrics and perturbations for the Seiberg-Witten equations to be the subset \(\Pi_E\subset \text{Met}(T(E/B))\times i\Omega^{2}(T(E/B))\) where
		\[\Pi_E:=\{(g,\eta):\star_b\eta_b=\eta_b\}.\] 
		The subset of \(\Pi\) for which the families Seiberg-Witten equations corresponding to the pair \((g,\eta)\) admit no reducible solutions is denoted \(\Pi_E^*\) and its complement \(\Pi_E\setminus \Pi_E^*\) called the wall is denoted \(\mathcal{W}\).
		
		The smooth family of metrics defines a smooth family of Levi-Civita connections on each \(TX_b\), so that for any smooth family of connections \(\{2A_b\}\) on \(\mathcal{L}_{\mathfrak{s}}\) where \(\mathcal{L}_{\mathfrak{s}}=\text{det}(W_b)\) and \(W_b\) is the spinor bundle induced by the fiberwise \(\text{spin}^c\) structure over \(X_b\), there is then an induced family of Dirac operators \(D_{A_b}:C^{\infty}(X_b,W^+_b)\rightarrow C^{\infty}(X_b,W^-_b)\). The families Seiberg-Witten moduli space \(\mathcal{M}(E,\mathfrak{s},g,\eta)\) is then defined as the solution set to the perturbed families Seiberg-Witten equations
		\begin{align}
			D_{A_b}\Phi_b&=0 \label{eqn:familiesSWeqn1fibres} \\
			\rho^{+_b}(F_{A_b}+\eta_b)&=(\Phi_b\Phi_b^*)_0 \label{eqn:familiesSWeqn2fibres}
		\end{align}
	  modulo the fibrewise action of the gauge groups \(\mathcal{G}_b=\text{Maps}(X_b,S^1)\).
	  The families moduli space is the disjoint union of the fibrewise moduli spaces
	  \[\mathcal{M}(E,\mathfrak{s},g,\eta):=\bigsqcup_{b\in B}\mathcal{M}_b.\]
	  with obvious projection map \(\mathcal{M}\longrightarrow B\) obtained from the disjoint union. For a generic family of perturbations \(\{\eta_b\}\), the moduli space is a smooth compact manifold of dimension \(\text{dim}(B)+d(X,\mathfrak{s})\) where \(d(X,\mathfrak{s})\) is the expected dimension of the unparametrised moduli space
	  \[d(X,\mathfrak{s})=\frac{\braket{c(L_{\mathfrak{s}})^2,[X]}-2\chi(X)-3\sigma(X)}{4}.\]
	  As in the unparametrised theory, generic choices of perturbations and metrics are cobordant provided they correspond to perturbations in the same connected component of \(\Pi\setminus\mathcal{W}\) called a \textit{chamber}. Hence, the families Seiberg-Witten invariant will only depend on a choice of such chamber. For \(b^+(X)>\text{dim}(B)+1\) there is only one connected component. However, the general situation is more complicated, where the set of chambers can be identified with \([B,\mathcal{P}]_f\), the fibrewise homotopy classes of sections into the period bundle \(\mathcal{P}\rightarrow B\) \cite{li_family_2001}.
	  Since \(H^{2,+}(X_b;i\R)\) has constant dimension over \(B\), it defines a vector bundle denoted \(H^+\), we shall assume that this vector bundle is orientable. One may then transport a relative orientation onto the families moduli space so that we may define the families Seiberg-Witten invariants as follows.
	  
	  Consider the subgroup of gauge group \(\mathcal{G}_b\) defined by
	  \[\mathcal{G}_{0,b}:=\{g\in \mathcal{G}_b:g=e^{if}, f:X_b\longrightarrow \R,\int_{X_b}f\text{vol}_{X_b}=0\}\]
	  since \(b_1(X)=0\) there is an exact sequence
	  \[0\rightarrow  \mathcal{G}_0\rightarrow \mathcal{G}\rightarrow S^1\rightarrow 0\]
	  Choose a generic perturbation so that it defines a chamber \(\phi\), let \(\widetilde{\mathcal{M}}\) be the moduli space factoring out by \(\mathcal{G}_0\) instead of the full gauge group fibrewise. This then gives a principal circle bundle \(\widetilde{\mathcal{M}}\longrightarrow\mathcal{M}\) since \(\mathcal{G}/\mathcal{G}_0\cong S^1\). Denote the associated complex line bundle by \(\mathcal{L}\longrightarrow \mathcal{M}\) and its first Chern class by \(y\in H^2(\mathcal{M};\Z)\). Since \(\mathcal{M}\) is equipped with a relative orientation there is a pushforward map \(\pi_*:H^k(\mathcal{M},\Z)\rightarrow H^{k-d}(B;\Z)\), where \(d\) is the fibrewise dimension of the moduli space. Note that the induced map on deRham cohomology corresponds to fibre integration. For \(n\ge 0\), the families Seiberg-Witten invariants are denoted
	  \[\FSW_n(E,\mathfrak{s},\phi)\in H^{2n-d}(B;\Z)\]
	  where \(\mathfrak{s}\) is the choice of families \(\text{spin}^c\) structure and \(\phi\) is a choice of chamber, defined by
	  \[\FSW_n(E,\mathfrak{s},\phi)=\pi_*(y^n).\]

	\section{A Computation for a General Class of K\"ahler Families}
	\label{sec:computation for general class of kahler families}
	
	In this section we demonstrate the computation of the families Seiberg-Witten invariant on smooth families \(X\hookrightarrow E\rightarrow B\) with a smoothly varying K\"ahler structure when \(b_1(X)=0\). We say a smooth family is \textit{K\"ahler} or \textit{has a smoothly varying K\"ahler structure}, if there is a K\"ahler structure on the vertical tangent bundle \(T(E/B)\). That is, there exist the following
	\begin{itemize}
		\item a metric \(g\) on \(T(E/B)\) (equivalently a smooth family of metrics on each \(TX_b\)).
		\item an almost-complex structure \(J\) on \(T(E/B)\) compatible with the metric \(g\) whose restriction to any fibre is integrable.
	\end{itemize}
	and the induced non-degenerate \((1,1)\)-form \(\omega\in C^{\infty}(E,\Lambda^2T^*(E/B))\) defined by
	\[\omega(v,w)=g(Jv,w)\]
	restricts to a closed form on each fibre. 
	
	To compute the invariant for such a family, we shall require an addition assumption that certain cohomology groups have constant dimension over the base of the family \(B\).
	Consider a smooth family of K\"ahler surfaces \(E\longrightarrow B\) with smooth K\"ahler structure,where \(B\) is compact and the fibres \(X_b\) are diffeomorphic to some compact K\"ahler surface \(X\) with \(b_1(X)=0\). Choose a family of metrics \(g_b\) and the family of self-dual 2-forms \(i\lambda\omega_b\) with \(\lambda>0\) sufficiently large so that \(i\lambda\omega_b\) does not lie on the wall for any \(b\in B\) as a family of perturbations, where the \(\omega_b\) are the K\"ahler forms obtained from the K\"ahler structure. Note that the existence of such a \(\lambda\) follows from the compactness of \(B\) by taking a finite trivialising cover. This choice of perturbation determines a chamber for the families Seiberg-Witten invariant which we call the \textit{K\"ahler chamber}.
	
	Since the family has a smoothly varying K\"ahler structure, there is a canonical \(\text{spin}^c\) structure on the vertical tangent bundle. Any other families \(\text{spin}^c\) structure \(\mathfrak{s}_L\) is obtained via tensoring by a hermitian line bundle \(L\) over \(E\), as such, we denote the families Seiberg-Witten invariants in the K\"ahler chamber corresponding to the \(\text{spin}^c\) structure determined from \(L\) by \(\FSW_n(E,L)\). This restricts to a line bundle over each \(X_b\) which we denote by \(L_b\), note that a choice of a family of connections for the families Seiberg-Witten equations amounts to a choice of a smooth family of connections \(A_b\) on \(L_b\). Since it is assumed that \(b_1(X)=0\), there is at most a one holomorphic structure on each \(L_b\), if it exists, it is precisely the one given by \(A_b\). Hence, we denote the cohomology groups for the sheaf of holomorphic sections of \(L_b\) with respect to the connection \(A_b\) by \(H^i(X_b,L_b)\). We shall make the following assumptions relating to the choice of line bundle \(L\).
	\begin{enumerate}
		\item For each \(b\in B\), the line bundle \(L_b\) has first Chern class \(c_1(L_b)\) which is represented by a \((1,1)\) form, hence the line bundles \(L_b\) are all holomorphic. \\
		\item The dimensions of \(H^i(X_b,L_b)\) for \(i=0,1,2\) are independent of \(b\).
	\end{enumerate}
	The first assumption ensures the line bundles \(L_b\) are holomorphic which is simply a necessary condition for non-zero Seiberg-Witten invariants, the second ensures that the families of vector spaces \(H^i(X_b,L_b)\) form locally trivial vector bundles over \(B\) 
	\[V^i\longrightarrow B\]
	where the fibres are \(H^i(X_b,L_b)\) and are of rank \(h^i(L)\), for simplicity we shall often denote the rank by \(h^i\) when \(L\) is understood. With these assumptions we may then proceed to prove the following theorem, which is a computation of the families Seiberg-Witten invariants.
	\begin{theorem}
		\label{thm:general computation of the families Seiberg-Witten invariant}
		Let \(X\hookrightarrow E\rightarrow X\) be a smooth family of compact K\"ahler surfaces with a smoothly varying K\"ahler structure and \(b_1(X)=0\) such that the following assumptions hold
		\begin{enumerate}
			\item For all \(b\in B\), the line bundle \(L_b\) has first Chern class \(c_1(L_b)\) which is represented by a \((1,1)\) form. \\
			\item The dimensions \(h^i\) of \(H^i(X_b,L_b)\) for \(i=0,1,2\) are independent of \(b\).
		\end{enumerate}
		Define \(\delta=m+n-h^0+1\), the families Seiberg-Witten invariants in the K\"ahler chamber are given by
		\begin{equation}
			\label{eqn:general computation of the families Seiberg-Witten invariant}
			\FSW_n(E,L)=\sum_{m=0}^{h^1-h^2+\rho_g}c_{h^1-h^2+\rho_g-m}(H^{2,0})\Gamma_{m,n}
		\end{equation}
		where \(\Gamma_{m,n}\in H^{2\delta}(B;\Z)\) is given by
		\begin{align}
			\Gamma_{m,n}&=(-1)^n\sum_{i=\max\{\delta-m,0\}}^{\delta}\sum_{j=0}^{\min\{i,m-\delta+i\}}s_j(V^2)c_{\delta-i}(V^1)s_{i-j}(V^0){h^1-h^2-\delta+i-j\choose m-\delta+i-j}.
		\end{align}
		for \(1\le h^0\le m+n+1\) and is zero otherwise.
		
		Furthermore, the \(\Gamma_{m,n}\) satisfy the following recursion relation
		\[\Gamma_{m,n+h^0} - c_1(V^0)\Gamma_{m,n+h^0-1} + c_2(V^0)\Gamma_{m,n+h^0-2}+\dots = 0.\]
	\end{theorem}
\begin{proof}
	Since \(B\) is compact, by taking a finite trivialising cover of the family and applying \Cref{thm:bradlow's theorem}, it follows that we can further choose \(\lambda\) sufficiently large so that \(\eta\) lies away from the wall and on each fibre \(X_b\), solutions to the Seiberg-Witten equations up to gauge equivalence will correspond to, up to gauge equivalence of complex line bundles, a holomorphic structure on the line bundle \(L_b\) with non-zero holomorphic section. It follows that fibrewise, the moduli space \(\mathcal{M}_b\) is given by
	\[\mathcal{M}_b=\frac{H^0(X_b,L_b)\setminus\{0\}}{\C^*}=\P(H^0(X_b,L_b))\]
	i.e. the projectivisation of the vector space \(H^0(X,L_b)\). Therefore, the families moduli space is the vector bundle
	\[\mathcal{M}=\P(V^0).\]
	Let \(\widetilde{\mathcal{M}}\) be the families moduli space obtained factoring by the family of reduced gauge groups \(\{\mathcal{G}_{0,b}\}_{b\in B}\) instead of the full gauge group. Since \(\mathcal{G}_b\cong \mathcal{G}_{0,b}\times S^1\), there is a principal circle bundle
	\[\widetilde{\mathcal{M}}\longrightarrow\mathcal{M}\]
	for which there is an associated line bundle over the total moduli space, denoted \(\mathcal{O}_{V^0}(-1)\), this restricts to the tautological bundle \(\mathcal{O}(-1)\) over each \(\P(V^0)_b\). Define \(y=c_1(\mathcal{O}_{V^0}(-1))\), and \(x=c_1(\mathcal{O}_{V^0}(1))\) where \(\mathcal{O}_{V^0}(1)\) is the dual of \(\mathcal{O}_{V^0}(-1)\), then \(y=-x\).
	
	Fix some integer \(k>2\) and a smooth reference connection \(A_0\), the moduli space can be obtained as the quotient \(\mathcal{M}=\widetilde{\mathcal{M}}/S^1\) where \(\widetilde{\mathcal{M}}\) is obtained by adding the gauge fixing condition \(d^*(A-A_0)=0\) and is the zero set of the Fredholm map \(f:V\rightarrow W\) between Hilbert bundles over \(B\) given by
	\[f(\alpha_b,\Phi_b)=(D_{A_{0,b}+\alpha_b}\Phi,d^+\alpha+\sigma((\Phi_b\Phi_b^*)_0)-F_{A_{0,b}}-i\pi\lambda\omega_b,d^*\alpha_b)\]
	acting fibrewise, where
	\[V=V_{\C}\oplus V_{\R} ,\qquad W=W_{\C}\oplus W_{\R}\]
	and \(V_{\C},W_{\C},V_{\R},W_{\R}\) have fibres
	\begin{alignat*}{2}
		&V_{\C}:=L^2_{k}(X_b,W^+_b),\qquad  &&V_{\R}:= iL^2_k(X_b,\Lambda^1T^*X_b) \\
		&W_{\C}:=L^2_{k-1}(X_b,W^-), \qquad &&W_{\R}:= iL^2_{k-1}(\Lambda^{2,+}T^*X_b)\oplus L^2_{k-1}(X_b,i\R)_0
	\end{alignat*}
	where \(L^2_q(X_b,\R)_0\) is the subspace satisfying \(\int_{X_b}f\text{vol}=0\). This is invariant under the remaining \(S^1\) action and hence descends to a map
	\[\overline{f}:V/S^1\longrightarrow W\]
	then the moduli space is the inverse image of the zero section
	\[\mathcal{M}=\overline{f}^{-1}(0).\]
	The complex Fredholm index of \(\overline{f}\) is the expected dimension of the unparametrised moduli space, that is \(\text{ind}_{\C}(\overline{f})=\chi(X,L)-p_g-1\). We may define a Hilbert bundle \(\mathcal{H}\rightarrow\mathcal{B}\) by taking the trivial bundle over \(V\) with fibres \(W\) and then applying the \(S^1\) action to the base, \(\overline{f}\) then defines a Fredholm section \(\sigma\) of this bundle for which the moduli space is \(\sigma^{-1}(0)\).
	
	We now wish to compute the families Seiberg-Witten invariants. However, the particular choice of perturbation is not necessarily generic, we shall apply the same technique as in \Cref{thm:general computation of the SW invariant on Kahler surfaces}. It suffices to show that the cokernels of \(d\overline{f}\) forms a vector bundle called the obstruction bundle, then a generic perturbation of the moduli space is isotopic to the zero set of a section of the obstruction bundle \cite{friedman_obstruction_1999}[Lemma 3.3]), the families Seiberg-Witten invariants are then given by inserting a factor of the Euler class of the obstruction bundle \(e(\text{Obs})\) where \(\text{Obs}=\text{coker}(\overline{f}):V\longrightarrow W\), that is
	\[\text{FSW}_n(E,L)=\pi_*(y^ne(\text{Obs})).\]
	
	We proceed to look at the linearisation of the families Seiberg-Witten equations. 
	Since the chosen perturbation is a (1,1)-form, without loss of generality we may take \(\varphi_2=0\). For a solution of the form \((A,\varphi_0,0)\) the linearised equations in the K\"ahler case are as follows
	\begin{align}
		-2i(d\alpha)_{\omega}-\text{Re}(\overline{\varphi}_0\tau_0)&=0 \label{eqn:KSW linearisation 1} \\
		\partialbar\tau_0+\partialbar^*\tau_2+\alpha^{0,1}\varphi_0&=0 \label{eqn:KSW linearisation 2}\\
		2(d\alpha)^{0,2}-\overline{\varphi}_0\tau_2&=0 \label{eqn:KSW linearisation 3}
	\end{align}
	acting on a triple \((\alpha,\tau_0,\tau_2)\in i\Omega^{1}(X,L)\oplus \Omega^{0,0}(X,L)\oplus \Omega^{0,2}(X,L)\).
	
	It is convenient to consider the following gauge fixing conditions, namely applying the Coulomb gauge fixing condition
	\begin{equation}
		d^*\alpha=0\label{eqn:Coulomb Gauge}
	\end{equation}
	this eliminates the action of all but the constant maps into the circle, so there is a remaining \(S^1\) action required to obtain the moduli space. This condition alongside the linearised Seiberg-Witten equations gives rise to an operator
	\begin{equation}
		\label{eqn:KSW linearised operator with coulomb gauge}
		\mathcal{D}_{A,\varphi}\begin{pmatrix}
			\tau_0 \\ \alpha \\ \tau_2
		\end{pmatrix}=\begin{pmatrix}
			-2i(d\alpha)_{\omega}-\text{Re}(\overline{\varphi}_0\tau_0) \\
			\partialbar\tau_0+\partialbar^*\tau_2+\alpha^{0,1}\varphi_0 \\
			2(d\alpha)^{0,2}-\overline{\varphi}_0\tau_2 \\
			d^*\alpha
		\end{pmatrix}
	\end{equation}
	the other gauge fixing condition of interest is the following
	\begin{equation}
		\label{eqn:KSW gauge fixing condition}
		d^*\alpha-i\braket{i\varphi_0,\tau_0}=0
	\end{equation}
	this asserts that \((\alpha,\tau_0,\tau_2)\) is orthogonal to the orbit of \((A,\varphi_0,0)\) under the action of the gauge group, completely eliminating the gauge group action. Equations \ref{eqn:KSW linearisation 1} and \ref{eqn:KSW gauge fixing condition} can equivalently be stated as 
	\begin{equation}
		2\partialbar\alpha_1-\overline{\varphi}_0\tau_0=0
	\end{equation}
	where \(\alpha_1=\alpha^{0,1}\).
	
	This leads us to define the linearised operator
	\begin{equation}
		\label{eqn:KSW linearised operator}
		\mathcal{\tilde{D}}_{A,\varphi}\begin{pmatrix}
			\tau_0 \\ \alpha_1 \\ \tau_2
		\end{pmatrix}=\begin{pmatrix}
			\partialbar^*\alpha_1-\overline{\varphi_0}\tau_2/2 \\
			\partialbar_B\tau_0+\partialbar_B^*\tau_0+\alpha_1\varphi_0\\
			\partialbar\alpha_1-\overline{\varphi_0}\tau_2/2
		\end{pmatrix}
	\end{equation}
	We have the following isomorphisms
	\[\text{ker}(df_{(A,\Phi)})\cong\text{ker}(\mathcal{D}_{A,\Phi}),\qquad\text{coker}(df)\cong\text{coker}(\mathcal{D}_{A,\Phi}).\]
	Consequently
	\begin{align*}
		\text{ker}(d\overline{f}_{[A,\Phi]})&\cong\text{ker}(\widetilde{\mathcal{D}}_{A,\Phi})\cong\text{ker}(\mathcal{D}_{A,\Phi})/\R(0,i\Phi) \\
		\qquad\text{coker}(d\overline{f}_{[A,\Phi]})&\cong\text{coker}(\widetilde{\mathcal{D}}_{A,\Phi})\cong\text{coker}(\mathcal{D}_{A,\Phi}).
	\end{align*}
%
	Recall for each fibre \(X_b\), \Cref{lem:mrowkas exact sequence} gives the following exact sequence of vector spaces.
	\[\begin{tikzcd}
		0 & {H^0(X_b,\mathcal{O})} & {H^0(X_b,L_b)} & {\text{ker}(\mathcal{\tilde{D}}_{A,\varphi,b})} \\
		& {H^1(X_b,\mathcal{O})} & {H^1(X_b,L_b)} & {\text{coker}(\mathcal{\tilde{D}}_{A,\varphi,b})} \\
		& {H^2(X_b,\mathcal{O})} & {H^2(X_b,L_b)} & 0
		\arrow[from=1-1, to=1-2]
		\arrow[from=1-2, to=1-3]
		\arrow[from=1-3, to=1-4]
		\arrow[from=1-4, to=2-2]
		\arrow[from=2-2, to=2-3]
		\arrow[from=2-3, to=2-4]
		\arrow[from=2-4, to=3-2]
		\arrow[from=3-2, to=3-3]
		\arrow[from=3-3, to=3-4]
	\end{tikzcd}\]
	and \(\mathcal{O}\) is the usual structure sheaf of \(X_b\).
	
	Since \(b_1(X)=0\), it follows that \(H^1(X_b,\mathcal{O})=0\), hence there is a reduction of the exact sequence to
	\[
	0 \rightarrow {H^1(X_b,L_b)} \rightarrow {\text{coker}(\mathcal{D}_{A,\varphi,b})} \rightarrow {H^2(X_b,\mathcal{O})} 
	\rightarrow {H^2(X_b,L_b)} \rightarrow 0
	\]
	Since the dimensions of \(H^i(X_b,L_b)\) are assumed to be constant over \(b\), it follows that the cokernels have constant dimension, hence we define vector bundles \(\text{coker}(\mathcal{D})\) with fibres \(\text{coker}(\mathcal{D}_{A,\varphi,b})\) and \(V^i\) with fibres \(H^i(X_b,L_b)\) over \(B\) as mentioned previously. We also define the vector bundle \(H^{2,0}\) to be the vector bundle with fibres \(H^2(X_b,\mathcal{O})\). The previous sequence then implies that the following sequence of smooth vector bundles over \(B\) is exact.
	
	\[
	0 \rightarrow {V^1} \rightarrow {\text{coker}(\mathcal{D})} \rightarrow {H^{2,0}} 
	\rightarrow {V^2} \rightarrow 0
	\]
	
	This induces an exact sequence of vector bundles over \(\widetilde{\mathcal{M}}\) by taking the relevant pullback of all the vector bundles involved. The remaining \(S^1\) action acts fibrewise, identifying cokernels via
	\[\text{coker}(\mathcal{D}_{A,\Phi})\cong \text{coker}(\mathcal{D}_{A,e^{i\theta}\Phi})\]
	yielding the obstruction bundle. The \(S^1\) action induces an action on \(L_b\) via rotation of complex numbers of unit length, further inducing an action on \(H^i(X_b,L_b)\) by multiplication of the same complex number. Meanwhile, the action on \(H^{2,0}\) is trivial. Thus, applying the \(S^1\) action results in tensoring by a factor of the hyperplane line bundle \(\mathcal{O}(1)\) for these terms. All the relevant maps in the exact sequence are also \(S^1\) equivariant, so applying the \(S^1\) action gives rise to the following well-defined exact sequence of vector bundles over \(\mathcal{M}\).
	\[
	0 \rightarrow {\pi^*V^1\otimes\mathcal{O}_{V^0}(1)} \rightarrow {\text{Obs}} \rightarrow {\pi^*H^{2,0}}
	\rightarrow {\pi^*V^2\otimes\mathcal{O}_{V^0}(1)} \rightarrow 0
	\]
	where \(\pi\) is the projection map \(\pi:\mathcal{M}\longrightarrow B\).

	
	It follows that the obstruction bundle as a complex vector bundle has rank given by \(\text{rank}(\text{Obs})=h^1-h^2+\rho_g\), the total Chern class of the obstruction bundle is also given by
	\begin{equation}
		\label{eqn:chern of Obs}
		c(\text{Obs})=c(\pi^*H^{2,0})c(\pi^*V^1\otimes\mathcal{O}_{V^0}(1))c(\pi^*V^2\otimes\mathcal{O}_{V^0}(1))^{-1}.
	\end{equation}
	 The top degree term, which coincides with the Euler class of the obstruction bundle, is given by the sum of all products which are cohomology classes of degree \(\text{rank}_{\R}(\text{Obs})=2(h^1-h^2+\rho_g)\), i.e.
	 \begin{equation}
	 	\label{eqn:expansion of the e(Obs) term}
	 	e(\text{Obs})=\sum_{m=0}^{h^1-h^2+\rho_g}c_{h^1-h^2+\rho_g-m}(\pi^*H^{2,0})\phi_{m}.
	 \end{equation}
	 where \(\phi_m\) is the \(2m\)-th degree term of the product \(c(\pi^*V^1\otimes\mathcal{O}_{V^0}(1))c(\pi^*V^2\otimes\mathcal{O}_{V^0}(1))^{-1}\). This can be made further explicit, since if \(E\) is a vector bundle of rank \(r\) and \(L\) is a line bundle, then the \(i\)-th Chern class is obtained by
	\begin{equation}
		\label{eqn:expansion of the chern classes under the tensor product of a line bundle}
		c_i(E\otimes L)=\sum_{j=0}^i{r-i+j\choose j}c_{i-j}(E)c_1(L)^j.
	\end{equation}
	Given a vector bundle \(E\rightarrow B\), one may define the \textit{total Segre class} by \(s(E):=c(E)^{-1}\) where \(c(E)=c_0(E)+c_1(E)+c_2(E)+\dots\) is the total Chern class of \(E\). The \(i\)\textit{-th Segre class} is then defined as \(2i\)-th graded piece in cohomology of the total Segre class, that is, \(s_i(E)\in H^{2i}(B;\Z)\) is defined so that \(s(E)=s_0(E)+s_1(E)+\dots\). The \(s_i\) can inductively be computed by the relations
	\begin{align*}
		c_0(E)&=s_0(E)=1 \\
		c_1(E)&=-s_1(E) \\
		c_2(E)&=s_1(E)^2-s^2(E) \\
		\vdots  \\
		c_r(E)&=-s_1(E)c_{n-1}(E)-s_2(E)c_{n-2}(E)-\dots -s_n(E)
	\end{align*}
where \(s_j(E)=0\) for \(j<0\). There is a formula for the Segre classes of the tensor product of a vector bundle with a line bundle, similar to that for the Chern classes as follows
	\begin{equation}
		\label{eqn:expansion of the segre classes under the tensor product of a line bundle}
		s_i(E\otimes L)=\sum_{j=0}^i(-1)^{i-j}{r+i-1\choose r-1+j}s_j(E)c_1(L)^{i-j}.
	\end{equation}
	Consequently, the \(c(\pi^*V^2\otimes\mathcal{O}_{V^0}(1))^{-1}\) factor in \Cref{eqn:chern of Obs} is the total Segre class, \(s(\pi^*V^2\otimes\mathcal{O}_{V^0}(1))\).
	It follows that we have
	\[c(\pi^*V^1\otimes\mathcal{O}_{V^0}(1))=\sum_{\ell=0}^{h^1}
	\sum_{j=0}^{\ell}{h^1-\ell+j \choose j}c_{\ell-j}(\pi^*V^1)x^j\]
	where \(x=c_1(\mathcal{O}_{V^0}(1))\), and similarly
	\[c(\pi^*V^2\otimes\mathcal{O}_{V^0}(1))^{-1}=\sum_{p=0}^{\infty}\sum_{i=0}^p(-1)^{p-i}{h^2+p-1 \choose h^2+i-1}s_i(\pi^*V^2)x^{p-i}\]
	hence we obtain the following explicit expression for \(\phi\)
	\begin{equation}
		\label{eqn:expansion of the phi_m's}
		\begin{split}
			\phi_{m}=\sum_{p=0}^{m}\sum_{j=0}^p\sum_{i'=0}^{m-p}\Biggl[(-1)^{p-j}&{h^2+p-1 \choose h^2+j-1}{h^1-m+p+i' \choose i'} \\ &s_j(\pi^*V^2)c_{m-p-i'}(\pi^*V^1)x^{p+i'-j}\Biggr].
		\end{split}
	\end{equation}
	
	The other required object to compute the families Seiberg-Witten invariants is \(\pi_*(y^k)\). Since \(y=c_1(\mathcal{O}_{V^0}(-1))=-x\), it suffices to understand \(\pi_*(x^k)\).
	
	Since the fibres of \(\P(V^0)\longrightarrow B\) are \(h^0-1\) complex dimensional, for \(0\le k<h^0-1\) we have
	\[\pi_*(x^k)=0\]
	
	To obtain \(\pi_*(x^{h^0-1})\), consider the Euler sequence for the projectivisation of a vector bundle, this gives an exact sequence of vector bundles over \(\mathcal{M}=\P(V^0)\).
	\[\begin{tikzcd}
		0 & \C & {\mathcal{O}_{V^0}(1)\otimes \pi^*V^0} & {T(\P(V^0)/B)} & 0
		\arrow[from=1-1, to=1-2]
		\arrow[from=1-2, to=1-3]
		\arrow[from=1-3, to=1-4]
		\arrow[from=1-4, to=1-5]
	\end{tikzcd}\]
	due to the triviality of \(\C\) as a vector bundle it follows that
	\[c_{h^0-1}(T(\P(V^0)/B))=c_{h^0-1}(\mathcal{O}_{V^0}(1)\otimes \pi^*V^0).\]
	Since \(\mathcal{O}_{V^0}(1)\otimes V^0\) is a rank \(h^0\) vector bundle, expanding this further gives
	\begin{equation}
		\label{eqn:moduli space vertical tangent bundle chern class relation }
		c_{h^0-1}(T(\P(V^0)/B))=\sum_{j=0}^{h^0-1}{j+1 \choose j}c_{h^0-1-j}(\pi^*V^0)x^j.
	\end{equation}
	 Observe that \(c_{h^0-1}(T(\P(V^0)/B))\) is simply the Euler class of vertical tangent bundle, since it fibrewise restricts to the Euler class of the tangent bundle of the fibres. Since \(\P(V^0)_b\) is fibrewise isomorphic to \(\C\P^{h^0-1}\) applying \(\pi_*\) gives the Euler characteristic of \(\C\P^{h^0-1}\), thus the left hand side after applying \(\pi_*\) is \(h^0\). By the projection formula it then follows that
	\[h^0=\sum_{j=0}^{h^0-1}{j+1\choose j}c_{h^0-1-j}(V^0)\pi_*(x^j)\]
	and only the \(j=h^0-1\) term is non-zero, hence
	\[h^0={h^0\choose h^0-1}\pi_*(x^{h^0-1})=h^0\pi_*(x^{h^0-1})\]
	and so
	\begin{equation*}	
		\pi_*(x^{h^0-1})=1.
	\end{equation*}
	
	
	To compute higher powers of \(x\), we shall obtain a relation which allows us to reduce any power of \(x\) larger than \(h^0-1\) to a polynomial in \(x\) of degree at most \(h^0-1\).
	
	Since the total Chern classes of \(T(\P(V^0)/B)\) and \(\mathcal{O}_{V^0}(1)\otimes \pi^*V^0\) are equal, in particular their \(h^0\)-th Chern classes are equal. Moreover, the vertical tangent bundle is a rank \(h^0-1\) vector bundle and \(\mathcal{O}_V(1)\otimes V\) is a rank \(h^0\) vector bundle, it follows that
	\[c_{h^0}(\mathcal{O}_V(1)\otimes V)=0.\]
	However, expanding out the Chern class of the tensor product via the formula we have
	\[0=x^{h^0}+c_1(\pi^*V^0)x^{h^0-1}+...+c_{h^0}(\pi^*V^0)\]
	which is the desired relation.
	
	
	With this expression, \(\pi_*(x^{h^0-1+j})\) can be computed for \(j>0\), we shall show a general recursion relation. Let \[\tau_j:=\pi_*(x^{j+h^0-1})\] with the initial conditions that \(\tau_0=1\) and \(\tau_j=0\) for \(j<0\).
	
	Since \(x^{h^0}+c_1(\pi^*V^0)x^{h^0-1}+...+c_{h^0}(\pi^*V^0))=0\), it follows that
	\begin{align*}
		0&=\pi_*(x^{j-1}(x^{h^0}+c_1(\pi^*V)x^{h^0-1}+...+c_{h^0}(\pi^*V^0))) \\
		&=\tau_j+c_1(V^0)\tau_{j-1}+\dots c_{h^0}(V^0)\tau_{j-h^0}
	\end{align*}
	thus
	\[\tau_j=-(c_1(V^0)\tau_{j-1}+\dots c_{h^0}(V^0)\tau_{j-h^0})\]
	this is precisely the same recursion relation and initial conditions as the Segre classes \(s_j(V^0)\), hence
	\begin{equation}
		\label{eqn:fibre integral of powers of x}
		\pi_*(x^{j+h^0-1})=s_j(V^0).
	\end{equation}

	We now proceed to use the computation for \(\pi_*(x^k)\) to obtain an expression of the invariant. It follows from the projection formula and \Cref{eqn:expansion of the e(Obs) term} that
	\begin{align*}
		\FSW_n(E,L)=\sum_{m=0}^{h^1-h^2+\rho_g}c_{h^1-h^2+\rho_g-m}(H^{2,0})\Gamma_{m,n}
	\end{align*}
	where \[\Gamma_{m,n}:=\pi_*(y^n\phi_{m}).\] 
	Since \(y\) satisfies \(y^{h^0} -c_1(V^0)y^{h^0-1} + c_2(V^0)y^{h^0-2}+\dots = 0\) we obtain the desired recursion relation in the theorem statement
	\[\Gamma_{m,n+h^0} - c_1(V^0)\Gamma_{m,n+h^0-1} + c_2(V^0)\Gamma_{m,n+h^0-2}+\dots = 0.\]
	To further simplify the expression for the families Seiberg-Witten invariant, we wish to simplify the expression for \(\Gamma_{m,n}\). From the expressions for \(\phi_m\) and \(\pi_*(x^k)\) (\cref{eqn:expansion of the phi_m's} and \cref{eqn:fibre integral of powers of x} respectively), this in addition to the fact that \(y=-x\), yields the following expression for \(\Gamma_{m,n}\)
	\begin{equation}
		\label{eqn:Gamma mn 1}
		\begin{split}
			\Gamma_{m,n}=\sum_{p=0}^{m}\sum_{j=0}^p\sum_{i'=0}^{m-p}\Biggl[&(-1)^{n+p-j}{h^2+p-1 \choose p-j}{h^1-m+p+i' \choose i'} \\ &s_j(V^2)c_{m-p-i'}(V^1)s_{p+i'+n-j-h^0+1}(V^0)\Biggr].
		\end{split}
	\end{equation}
	
	
	
	
	
	
	
	We introduce the quantity \(\delta:=m+n-h^0+1\) which keeps track of the degree of \(\Gamma_{m,n}\). For \(\Gamma_{m,n}\) to be non-zero, we require \(\delta\ge 0\) or equivalently \(m+n\ge h^0-1\).
	
	We may adjust the indices in \Cref{eqn:Gamma mn 1} so that \(i'\) begins at \(m-p-\delta\) and \(j\) ends at \(\delta\). To see this, observe that \(0\le i'\le m-p\), implies that \(p+i'+n-j\le m+n-j\), but \(m+n=h^0-1+\delta\) and so \(p+i'+n-j\le h^0-1+\delta-j\). Since \(\pi_*(x^k)\) is non-zero only if \(k\ge h^0-1\), the only non-zero terms which survive in the sum occur when \(m-p-\delta\le i'\le m-p\) and \(0\le j\le \delta\). Moreover, since \({a\choose b}=0\) for \(a\ge0,b<0\), even if \(m-p-\delta<0\), we can start the \(i'\) index at \(m-p-\delta\). Similarly, the \(j\) index may end at \(\delta\). Hence \(\Gamma_{m,n}\) can be written as
	\begin{equation*}
		\begin{split}
			\Gamma_{m,n}=(-1)^n\sum_{p=0}^m\sum_{i'=m-p-\delta}^{m-p}\sum_{j=0}^{\delta}\Biggl[&(-1)^{p-j}{h^2+p-1\choose p-j}{h^1-m-p+i'\choose i'} \\
			&s_j(V^2)c_{m-p-i'}(V^1)s_{p+i'+n-j-h^0+1}(V^0)\Biggr]
		\end{split}
	\end{equation*}
	set \(i=i'-m+p+\delta\) then the sum over \(i'\) can be rewritten as a sum over \(i\), since the range of values for the \(j\) and \(i\) indices is independent of \(p\), we may freely change the position of the sum over \(p\), this in conjunction with the identity that \(\binom{a}{b}=(-1)^b\binom{b-a-1}{b}\) applied to the first binomial coefficient yields the following expression
	\begin{equation*}
		\begin{split}
			\Gamma_{m,n}=(-1)^n\sum_{i=0}^{\delta}\sum_{j=0}^{\delta}\sum_{p=0}^m&\Biggl[{-h^2-j\choose p-j}{h^1-\delta+i\choose m-\delta+i-p} \\
			&s_j(V^2)c_{\delta-i}(V^1)s_{i-j}(V^0)\Biggr].
		\end{split}
	\end{equation*}
	Since the terms from the sum over \(p\) are zero unless \(0\le p\le m-\delta+i\) and also \(p\ge j\), the above expressions can be rewritten as
	\begin{equation*}
		\begin{split}
			\Gamma_{m,n}=(-1)^n\sum_{i=0}^{\delta}\sum_{j=0}^{\delta}\Biggl[&s_j(V^2)c_{\delta-i}(V^1)s_{i-j}(V^0) \\ &\sum_{p=j}^{m-\delta+i}{-h^2-j\choose p-j}{h^1-\delta+i\choose m-\delta+i-p}\Biggr]
		\end{split}
	\end{equation*}
	reindexing the sum over \(p\) as \(p'=p-j\) then gives
	\begin{equation*}
		\begin{split}
			\Gamma_{m,n}=(-1)^n\sum_{i=0}^{\delta}\sum_{j=0}^{\delta}&\Biggl[s_j(V^2)c_{\delta-i}(V^1)s_{i-j}(V^0) \\
			&\sum_{p'=0}^{m-\delta+i-j}{-h^2-j\choose p'}{h^1-\delta+i\choose m-\delta+i-j-p'}\Biggr]
		\end{split}.
	\end{equation*}
	The second binomial coefficient is zero unless \(m-\delta+i-j-p'\ge 0\), so for \(\Gamma_{m,n}\) to be non-zero, we require \(\delta-m+i\ge 0\), in such a case \(i\ge\max\{m-\delta,0\}\). When this occurs \(\delta-m+i\ge 0\) and we must also have \(\delta-m+i-j\ge 0\) for \(\Gamma_{m,n}\) to be non-zero. Furthermore, \(s_{i-j}(V^0)=0\) for \(i\le j\), thus \(j\le \min\{i,m-\delta+i\}\) and the sum may be readjusted to
	\begin{equation*}
		\begin{split}
			\Gamma_{m,n}=(-1)^n\sum_{i=\max\{\delta-m,0\}}^{\delta}\sum_{j=0}^{\min\{i,m-\delta+i\}}&\Biggl[s_j(V^2)c_{\delta-i}(V^1)s_{i-j}(V^0) \\
			&\sum_{p'=0}^{m-\delta+i-j}{-h^2-j\choose p'}{h^1-\delta+i\choose m-\delta+i-j-p'}\Biggr]
		\end{split}.
	\end{equation*}
	With the above restrictions on \(i\) and \(j\), we have \(m-\delta+i-j\ge0\), so applying the Vandermonde-Chu identity gives.
	\begin{align*}
		\Gamma_{m,n}&=(-1)^n\sum_{i=\max\{\delta-m,0\}}^{\delta}\sum_{j=0}^{\min\{i,m-\delta+i\}}s_j(V^2)c_{\delta-i}(V^1)s_{i-j}(V^0){h^1-h^2-\delta+i-j\choose m-\delta+i-j}
	\end{align*}
which is the desired formula.
\end{proof}

\begin{remark}
	Note that \(c_{\delta-i}(V^1)\) is only non-zero for \(0\le \delta-i\le h^1\), so the first sum in the expression for \(\Gamma_{m,n}\) may actually be restricted to begin from \(i=\max\{\delta-m,\delta-h^1,0\}\), although we do not include this in the above theorem so that the expression is more compact.
\end{remark}
	
	\begin{remark}
		The computation of \Cref{thm:general computation of the families Seiberg-Witten invariant} indeed reduces to the result of \Cref{thm:general computation of the SW invariant on Kahler surfaces} when \(B\) is a point. In such a case the vector bundles \(H^2,V^i\) are all trivial so the only possibly non-vanishing invariant occurs when \(m=h^1-h^2+\rho_g\), \(i=j=\delta=0\) and \(n=\rho_g+1-\chi(L)\), i.e. the ordinary Seiberg-Witten invariant, which is given by \(\Gamma_{h^1-h^2+\rho_g,\rho_g+1-\chi(L)}\). It is clear from the formula in \Cref{thm:general computation of the families Seiberg-Witten invariant} that this recovers the result in the unparametrised case.
	\end{remark}
%
%
%

\section[The Families Invariants for the Projectivisation Family and Family with Fibres \(\C\P^1\times\C\P^1\)]{The Families Seiberg-Witten Invariants for some families with fibres \(\C\P^2\) and \(\C\P^1\times\C\P^1\)}
\label{sec:examples 1}
	With a general computation for the families Seiberg-Witten invariants in hand, we shall now apply it to example cases of K\"ahler families. The first K\"ahler family we consider is the projectivisation of a rank 3 complex vector bundle. Let \(\widetilde{\pi}:V\longrightarrow B\) be a complex rank 3 vector bundle over some compact base space \(B\). We may define its projectivisation \(\pi:\P(V)\longrightarrow B\) by taking the fibres to be the projectivisation of the fibres of \(V\), that is \(\P(V)_b:=\P(V_b)\). 
	
	\begin{proposition}
		\label{prop:projectivisation family is a kahler family}
		Let \(V\longrightarrow B\) be a complex rank 3 vector bundle where \(B\) is compact. Then the projectivisation family \(\pi:\P(V)\longrightarrow B\) is a smooth K\"ahler family.
	\end{proposition}
	
	\begin{proof}
		
		Since each \(V_b\) is a 3 dimensional vector space, \(V_b\cong \C^3\) and hence the fibres of \(\P(V)\) are diffeomorphic to \(\C\P^2\), hence this is a family of \(\C\P^2\)'s. This diffeomorphism also preserves the natural orientations defined on the fibres and \(\C\P^2\), consequently, this is a smooth family. The cohomology ring of \(\C\P^2\) is well known with \(b^+=1\) and \(b_1(\C\P^2)=0\) so it remains to show this family has a K\"ahler structure. 
		
		Choose a Hermitian metric on \(V\), there is then a naturally induced metric on the tautological line bundle \(\mathcal{O}_V(1)\longrightarrow \P(V)\). Consider the Chern curvature with respect to this metric, since the restriction of \(\mathcal{O}_V(1)\) to the fibres of \(\P(V)\) is isomorphic to the tautological line bundle over \(\C\P^2\) and the Chern curvature of any metric on the hyperplane line bundle over \(\C\P^2\) is the closed positive \((1,1)\) form induced by the Fubini-Study metric. Consequently, the Chern curvature of the metric on \(\mathcal{O}_V(1)\) must be a closed \((1,1)\) form which is positive along the vertical tangent bundle, hence the projectivisation family is a smooth K\"ahler family.
	\end{proof}
We now present a computation for the families Seiberg-Witten invariants for the projectivisation family. We have -Hirsch theorem
\[H^2(E;\Z)\cong H^2(B;\Z)\oplus \Z\]
where the isomorphism is given by
\[H^2(B;\Z)\oplus \Z\ni(L,m)\mapsto \mathcal{O}_V(m)\otimes \pi^*L\]
consequently, any line bundle over \(\P(V)\) is of the form:
\[\mathcal{O}_V(m)\otimes \pi^*L\]
for some line bundle \(L\longrightarrow B\). The following theorem contains the computation of the families Seiberg-Witten invariants for the \(\text{spin}^c\) structures determined by such line bundles.
\begin{theorem}
	\label{thm:computation of the FSW for the projectivisation family}
	Let \(\pi:V\rightarrow B\) be a complex vector bundle of rank \(3\) and \(\P(V)\rightarrow B\) be the corresponding projectivisation family and \(S^k(V^*)\) be the \(k\)-th symmetric power of \(V^*\). Then, the families Seiberg-Witten invariants in the K\"ahler chamber for the \(\text{spin}^c\) structure obtained by twisting the canonical structure by \(\mathcal{O}_V(k)\otimes\pi^*L\) are given by
	\begin{equation}
		\FSW_n(E,\mathcal{O}_V(k)\otimes\pi^*L)=(-1)^ns_{n-h^0+1}(S^k(V^*)\otimes L)
	\end{equation}
	where \(h^0=\binom{2+k}{k}=(k+1)(k+2)/2\), when \(k\ge 0\) and are zero otherwise.
\end{theorem}
\begin{remark}
	The expression of the previous theorem can be made more explicit by expanding \(s_{n-h^0+1}(S^k(V^*)\otimes L)\) via \Cref{eqn:expansion of the segre classes under the tensor product of a line bundle}, denoting the invariants corresponding to the line bundle \(\mathcal{O}_V(k)\otimes\pi^*L\) by \(\FSW_n(E,L,k)\) we have
	\begin{equation}
		\FSW_n(E,L,k)=\sum_{i=0}^{n-h^0+1}(-1)^{h^0-1+i}\binom{n}{h^0-1+i}s_i(S^k(V^*))c_1(L)^{n-h^0+1-i}
	\end{equation}
	the only obstruction to making this formula even more explicit is obtaining computing the Segre classes of \(S^k(V^*)\), although this is possible in principle, it becomes increasingly difficult and quite cumbersome for larger and and larger \(k\).
\end{remark}
\begin{proof}[Proof of Theorem 5.2]
 Recall that we are interested in the vector bundles \(V^i\) with fibres over \(b\in B\) being the cohomology groups \(H^i(X_b,L_b)\), denote these cohomology bundles by
	\[V^j=H^j(\P(V),\mathcal{O}_V(k)\otimes \pi^*L)\]
	where \(\pi\) is the projection map of the projectivisation family \(\P(V)\rightarrow B\) induced by \(\widetilde{\pi}\). When \(k<0\) the line bundles \(\mathcal{O}(k)\) over \(\C\P^2\) do not admit global sections, hence \(H^0(\C\P^2,\mathcal{O}(k))=0\), this fiberwise calculation implies that
	\[H^0(\P(V),\mathcal{O}_V(k))=0\]
	and so \(h^0=0\) when \(k<0\). Since the dimension of the fiberwise moduli spaces is \(h^0-1\), they are empty for \(k<0\), hence the families Seiberg-Witten invariant is zero. It is then sufficient to consider the cohomology groups for \(k\ge 0\).
	
	By the Serre spectral sequence, we have for \(k\ge 0\)
	\[H^j(\P(V),\mathcal{O}_V(k))\cong
	\begin{cases}
		S^k(V^*) & j=0, \\
		0 & j>0.
	\end{cases}\]
	It follows from the projection formula that after twisting by the pullback of a line bundle \(L\) over \(B\)
	\[V^j=H^j(\P(V),\mathcal{O}_V(k)\otimes \pi^*L)\cong
	\begin{cases}
		S^k(V^*)\otimes L & j=0 \\
		0 & j>0
	\end{cases}.\]
	It immediately follows that \(h^1=h^2=0\). Moreover \(h^0\) is just the rank of \(S^k(V^*)\). Since \(V\) is a 3 dimensional complex vector bundle
	\begin{equation*}
		\text{rank}(S^k(V^*))=\binom{2+k}{k}=\frac{(k+1)(k+2)}{2}
	\end{equation*}
	

	We now proceed to obtain a more explicit expression for the families Seiberg-Witten invariants for the projectivisation family. Recall from \Cref{thm:general computation of the families Seiberg-Witten invariant} that
	\[\FSW_n(E,\mathcal{O}_V(k)\otimes\pi^*L)=\sum_{m=0}^{h^1-h^2+\rho_g}c_{h^1-h^2+\rho_g-m}(H^{2,0})\Gamma_{m,n}\]
	where \(H^2\) is the bundle with fibres \(H^2(X_b,\mathcal{O})\). It is a well known fact that \(b^+(\C\P^2)=1\), consequently \(\rho_g(\C\P^2)=0\), hence \(h^1-h^2+\rho_g=0\) and \(m\) is forced to be zero, thus
	\[\FSW_n(E,\mathcal{O}_V(k)\otimes\pi^*L)=\Gamma_{0,n}\]
	and \(\delta=n-h^0+1\). Applying the expression for \(\Gamma_{m,n}\) in \Cref{thm:general computation of the families Seiberg-Witten invariant} gives
	\begin{equation*}
		\FSW_n(E,\mathcal{O}_V(k)\otimes\pi^*L)=(-1)^n\sum_{i=\max\{\delta,0\}}^{\delta}\sum_{j=0}^{\min\{i,-\delta+i\}}s_j(V^2)c_{\delta-i}(V^1)s_{i-j}(V^0).
	\end{equation*}
	Recall that \(V^1\) and \(V^2\) are zero bundles, so the only surviving terms occur when \(j=0,i=n-h^0+1\). This results in the following computation of the families Seiberg-Witten invariants for the projectivisation family
	\begin{equation}
		\label{eqn:simple expression for the FSW of the projectivisation family}
		\FSW_n(E,\mathcal{O}_V(k)\otimes\pi^*L)=(-1)^ns_{n-h^0+1}(V^0)
	\end{equation}
	yielding the result.
\end{proof}

	%
	%
	%
	%
	We now investigate another family, the analysis here proceeds similarly to the previous example. Suppose that \(\pi_1:V_1\longrightarrow B\) and \(\pi_2:V_2\longrightarrow B\) are rank 2 complex vector bundles over \(B\) and consider their fibre product \(\Pi:\P(V^1)\times_B \P(V^2)\longrightarrow B\). This also has a K\"ahler structure, since as in \Cref{thm:computation of the FSW for the projectivisation family} we may endow \(\P(V^1)\) and \(\P(V^2)\) with K\"ahler structures which induces one on the product and consequently the fibre product. The Serre spectral sequence gives
	\[H^2(E;\Z)\cong H^2(B;\Z)\times \Z\times \Z\]
	where the map sends \(H^2(B;\Z)\times \Z\times \Z\ni(L,m,n)\mapsto \mathcal{O}_{V_1}(k)\otimes \mathcal{O}_{V_2}(\ell)\otimes \Pi^*L.\) Hence, we have the following theorem
	\begin{theorem}
		\label{thm:computation of the FSW for the CP1xCP1 family}
		Let \(\pi:V_1,V_2\rightarrow B\) be complex vector bundles of rank \(2\), \(S^k(V_i^*)\) be the \(k\)-th symmetric power of \(V_i^*\) and \(\P(V_1)\times_B\P(V_1)\rightarrow B\) be the family obtained from the fibre product of their projectivisations. Then the families Seiberg-Witten invariants in the K\"ahler chamber for the \(\text{spin}^c\) structure obtained by twisting the canonical structure by \(\mathcal{O}_{V_1}(k)\otimes\mathcal{O}_{V_2}(\ell)\otimes\Pi^*(L)\) are given by
		\begin{equation}
			\FSW_n(E,\mathcal{O}_{V_1}(k)\otimes\mathcal{O}_{V_2}(\ell)\otimes \Pi^*L)=(-1)^n\text{s}_{n-h^0+1}(S^k(V_1^*)\otimes S^{\ell}(V_2^*)\otimes L)
		\end{equation}
		for \(k,\ell\ge 0\), where \(h^0=(1+k)(1+\ell)\) and are zero otherwise.
	\end{theorem}
	Similar to the previous example, this formula can be made more explicit via the formula for the Chern classes of products of vector bundles and expressions for the Chern classes of \(S^k(V_1^*)\) and \(S^\ell(V_2^*)\), but identically, it also becomes increasingly cumbersome as \(k,\ell\) become larger to compute the Segre classes of the symmetric powers of \(V_1\) and \(V_2\).
	\begin{proof}[Proof of Theorem 5.4]
		As in \Cref{thm:computation of the FSW for the projectivisation family} unless \(k,\ell\ge 0\) we have \(h^0=0\) due to a lack of sections., yielding a zero Seiberg-Witten invariant. In the case when \(k,\ell\ge 0\), one may use the computation of the cohomology of line bundles as in the previous section on projective bundles and a K\"unneth formula for sheaf cohomology to obtain that the cohomology bundles are given by
		\[H^j(E;\mathcal{O}_{V_1}(k)\otimes \mathcal{O}_{V_2}(\ell))=\begin{cases}
			S^k(V_1^*)\otimes S^{\ell}(V_2^*) & j=0 \\
			0 & j>0
		\end{cases}\]
		the general case when twisting by a line bundle involves in tensoring by the line bundle \(L\) via the projection formula as before. Consequently, \(h^1=h^2=0\) and \(h^0=(1+k)(1+\ell)\), and by identical reasoning to that in \Cref{thm:computation of the FSW for the projectivisation family} we obtain that
		\[FSW_n(E,\mathcal{O}_{V_1}(k)\otimes\mathcal{O}_{V_2}(\ell)\otimes \Pi^*L)=(-1)^ns_{n-h^0+1}(V^0)\]
		where \(V^0=S^k(V_1^*)\otimes S^{\ell}(V_2^*)\otimes L\) from which the result follows.
	\end{proof}

	\section[The Families Invariants for Families of Blowups of K\"ahler Surfaces]{The Families Seiberg-Witten Invariants for Families of Blowups of K\"ahler Surfaces}
	\label{sec:families blowup}
	Given a K\"ahler surface \(X\) and \(x\in X\), one can construct the blowup of \(X\) at \(x\), denoted \(\text{Bl}_x(X)\). There is a natural projection map \(p:\text{Bl}_x(X)\rightarrow X\), let \(E=p^{-1}(x)\) be the exceptional divisor and \(\mathcal{O}(E)\) to be the line bundle over \(\text{Bl}_x(X)\) corresponding to the exceptional divisor. There are a variety of K\"ahler families one can obtain from this construction. Of primary interest shall be the \textit{universal blowup family}, this is a smooth family \(\pi:Z\longrightarrow X\) with base \(X\) and fibre at \(x\in X\) diffeomorphic to \(\text{Bl}_x(X)\). This is a special case of a more general construction of a sequence of spaces \(X_{\ell}\) with fibres diffeomorphic to the blowup of \(X\) at \(\ell\) possibly non-distinct points as in \cite{liu_family_2000}. The universal blowup family may be constructed as follows, let \(\Delta:X\longrightarrow X\times X\) be the diagonal section \(x\mapsto (x,x)\) and consider \(Z:=\text{Bl}_{\Delta(X)}(X\times X)\), the blowup of \(X\times X\) along the image of the diagonal section, this has a natural projection map \(\Pi:Z\longrightarrow X\times X\). We then define the universal blowup family to be the smooth fibre bundle \(\pi:Z\longrightarrow X\) where \(\pi=\Pi\circ p_2\) and \(p_2:X\times X\longrightarrow X\) is the projection onto the right factor. 

Given \(x\in X\), a fibre of the universal blowup family is \(\pi^{-1}(x)=\text{Bl}_x(X)\times\{x\}\cong \text{Bl}_x(X)\), so it is indeed a family with fibres diffeomorphic to the blowup of \(X\) at \(x\) and it satisfies a universal property, namely if \(Z'\longrightarrow B\) is another family where the fibre over \(b\in B\) is \(\text{Bl}_{f(b)}(X)\) for some smooth map \(f:B\rightarrow X\), then \(Z'\) is the pullback of \(\pi:Z\longrightarrow X\) under \(f:B\longrightarrow X\). Consequently, to compute the families Seiberg-Witten invariants of such blowup families specified by a smooth map \(f\), one simply needs to compute the pullback of the invariants of the universal blowup family. Because of this, we seek to compute the families Seiberg-Witten invariants of the universal blowup family when \(X\) is simply connected, first, consider the following lemma.
\begin{lemma}
	\label{prop:H2 of families with simply connected base and fibre}
	Let \(X\hookrightarrow E\longrightarrow B\) be a family of 4-manifolds with \(X\) and \(B\) both simply-connected and \(H^3(B;\Z)=0\) then
	\[H^2(X;\Z)\cong H^2(B;\Z)\oplus H^2(X;\Z)\]
\end{lemma}
\begin{proof}
	Since \(X\) is simply-connected, by the Hurcewicz theorem, \(H_1(X;\Z)=0\), then by Poincar\'e duality \(H^3(X;\Z)=0\). Since \(X\) is compact its homology groups are finitely generated so by the universal coefficient theorem, it follows that \(H^1(X;\Z)\cong \text{Hom}(H_1(X;\Z),\Z)\oplus T_0\) where \(T_0\) is the torsion subgroup of \(H_0(X;\Z)\), but \(X\) is connected so \(H_0(X;\Z)=\Z\) which has no torsion. Hence \(H^1(X;\Z)=0\). Since \(B\) is simply-connected, the action of \(\pi_1(B)\) on the fibres is trivial and the \(E^2\) page of the Leray-Serre spectral sequence consists of \(E^2_{p,q}=H^p(B;H^q(X;\Z))\) and abuts to \(H^{p+q}(E;\Z)\).
	
		
		
		
	
	The differentials all map to or out of 0 so the entries of the \(E^3\) page are identical to the \(E^2\) page and the differentials of the \(E^r\) page for \(r\ge 4\) are clearly all zero, hence there is an exact sequence
	\[0\rightarrow H^2(B;\Z)\rightarrow H^2(E;\Z)\rightarrow H^0(B;H^2(X;\Z))\rightarrow H^3(B;Z)\rightarrow 0\]
	but \(H^0(B;H^2(X;\Z))\) is just \(H^2(X;\Z)\) since \(B\) is connected and \(H^3(B;Z)\) is assumed to be zero. The exact sequence then reduces to
	\[H^2(B;\Z)\rightarrow H^2(E;\Z)\rightarrow H^0(B;H^2(X;\Z))\rightarrow 0.\]
	Again by the universal coefficient theorem we have \(H^2(X;\Z)\cong\text{Hom}(H_2(X;\Z),\Z)\oplus T_1\) where \(T_1\) is the torsion subgroup of \(H_1(X;\Z)\), but \(T_1=0\) since \(\pi_1(X)=0\). Since \(\text{Hom}(H_2(X;\Z),\Z)\) is free, it follows that \(H^2(X;\Z)\) is free and thus the above exact sequence splits and so
	\[H^2(X;\Z)\cong H^2(B;\Z)\oplus H^2(X;\Z).\]
\end{proof}
One may use this result to obtain the following.
\label{cor:line bundles on the blowup family}
\begin{corollary}
	Let \(X\) be a simply connected K\"ahler surface and \(\pi:Z\rightarrow X\) be the universal blowup family, fix some \(x\in X\) and let \(p:\text{Bl}_x(X)\rightarrow X\) be the natural projection map from the blowup of \(X\) at \(x\in X\), we then have the following isomorphism
	\[H^2(Z;\Z)\cong H^2(X;\Z)\oplus H^2(X;\Z)\oplus \Z\]
	given by \((L_1,L_2,k)\mapsto \pi^*L_1\otimes p^*L_2\otimes \mathcal{O}(kE)\)
\end{corollary}
\begin{proof}
	Note that for the blowup at a point \(x\in X\) with projection map \(p:\text{Bl}_x(X)\rightarrow X\) one has
	\[H^2(\text{Bl}_x(X);\Z)\cong H^2(X;\Z)\oplus\Z\]
	with the isomorphism specified by \((L,m)\mapsto p^*L\otimes\mathcal{O}(kE)\) where \(k\in \Z\) and \(E\) is the exceptional divisor. Since \(X\) is simply connected, applying \Cref{prop:H2 of families with simply connected base and fibre} one obtains for the universal blowup family \(H^2(Z;\Z)\cong H^2(X;\Z)\oplus H^2(X;\Z)\oplus \Z\) with \((L_1,L_2,k)\mapsto \pi^*L_1\otimes p^*L_2\otimes \mathcal{O}(kE)\).
\end{proof}	
	The restriction \(\pi^*L_1\otimes p^*L_2\otimes \mathcal{O}(kE)\) to a fibre \(Z_x\) is simply \(p^*L_2 \otimes \mathcal{O}(kE)\). Because of \Cref{cor:line bundles on the blowup family}, where simplification is required, we shall denote the corresponding families Seiberg-Witten invariants as \(\FSW_n(Z,L_1,L_2,k)\).
We now present a computation of the families Seiberg-Witten invariants for the universal blowup family when \(X\) is simply connected. 
	\begin{theorem}
		\label{thm:computation of the FSW for the universal blowup family} 
		Let \(X\) be a simply connected compact K\"ahler surface and \(\pi:Z\rightarrow X\) be the universal blowup family and assume the the dimensions \(h^i\) of the cohomology groups \(\{H^i(Z_b,(p^*L_2\otimes \mathcal{O}(kE)))\}_{b\in B}\) are constant over \(b\in B\) for \(i=0,1,2\) so that they define vector bundles \(W^i\) of rank \(h^i\), the families Seiberg-Witten invariants in the K\"ahler chamber for the \(\text{spin}^c\) structure obtained by twisting the canonical structure by the \(\pi^*L_1\otimes p^*L_2\otimes \mathcal{O}(kE)\) are given by the following formulae:
		
		when \(\delta=0\)
		\begin{equation}
			\label{eqn:FSW for the universal blowup family for delta=0}
			\FSW_n(Z,L_1,L_2,k)=(-1)^n\binom{h^1-h^2}{h^1-h^2+\rho_g}
		\end{equation}
		when \(\delta=1\), \(h^1-h^2+\rho_g=0\),
		\begin{equation}
			\label{eqn:FSW for the universal blowup family for delta=1, binom coeff denom is 0}
			\begin{split}
				\FSW_n(Z,L_1,L_2,k)=(-1)^n(s_1(W^0)-h^0c_1(L_1))
			\end{split}
		\end{equation}
		when \(\delta=1\), \(h^1-h^2+\rho_g\ge 1\),
		\begin{equation}
			\label{eqn:FSW for the universal blowup family for delta=1}
			\begin{split}
				\FSW_n(Z,L_1,L_2,k)&=(-1)^n(c_1(W^1)+h^1c_1(L_1))\binom{h^1-h^2-1}{h^1-h^2+\rho_g-1} \\
				&+(-1)^n(s_1(W^0)-h^0c_1(L_1))\binom{h^1-h^2}{h^1-h^2+\rho_g} \\
				&+(-1)^n(s_1(W^2)-h^2c_1(L_1))\binom{h^1-h^2-1}{h^1-h^2+\rho_g-1}
			\end{split}
		\end{equation}
		when \(\delta=2\), \(h^1-h^2+\rho_g=0\),
		\begin{equation}
			\label{eqn:FSW for the universal blowup family for delta=2, binom coeff denom is 0}
			\begin{split}
				\FSW_n(Z,L_1,L_2,k)=(-1)^{n}s_2(W^0\otimes L_1)
			\end{split}
		\end{equation}		
		when \(\delta=2\), \(h^1-h^2+\rho_g=1\), and
		\begin{equation}
			\label{eqn:FSW for the universal blowup family for delta=2, binom coeff denom is 1}
			\begin{split}
				\FSW_n(Z,L_1,L_2,k)&=+(-1)^{n}c_1(W^1\otimes L_1)s_1(W^0\otimes L_1) \\
				&+(-1)^{n}s_2(W^0\otimes L_1)\binom{h^1-h^2}{1} \\
				&+(-1)^{n}s_1(W^2\otimes L_1)s_1(W^0\otimes L_1)
			\end{split}
		\end{equation}
		when \(\delta=2\), \(h^1-h^2+\rho_g\ge 2\)
		\begin{equation}
			\label{eqn:FSW for the universal blowup family for delta=2}
			\begin{split}
				\FSW_n(Z,L_1,L_2,k)&=(-1)^{n}c_2(W^1\otimes L_1)\binom{h^1-h^2-2}{h^1-h^2+\rho_g-2} \\
				&+(-1)^{n}c_1(W^1\otimes L_1)s_1(W^0\otimes L_1)\binom{h^1-h^2-1}{h^1-h^2+\rho_g-1} \\
				&+(-1)^{n}s_1(W^2\otimes L_1)c_1(W^1\otimes L_1)\binom{h^1-h^2-2}{h^1-h^2+\rho_g-2} \\
				&+(-1)^{n}s_2(W^0\otimes L_1)\binom{h^1-h^2}{h^1-h^2+\rho_g} \\
				&+(-1)^{n}s_1(W^2\otimes L_1)s_1(W^0\otimes L_1)\binom{h^1-h^2-1}{h^1-h^2+\rho_g-1} \\
				&+(-1)^{n}s_2(W^2\otimes L_1)\binom{h^1-h^2-2}{h^1-h^2+\rho_g-2}
			\end{split}
		\end{equation}
		when \(\delta=2\), where \(\delta=\rho_g+1-\chi(Z,L)+n\) and \(\chi(Z,L)=h^0-h^1+h^2\). The invariants are zero otherwise.
	\end{theorem}
Before we begin the computation of the formulae mentioned on the previous page, we present the following lemma.
\begin{lemma}
	\label{lem:H2 is trivial for the blowup bundle}
	The bundle \(H^{2,0}\) with fibres \(H^2(\text{Bl}_x(X),\mathcal{O}_{\text{Bl}_x(X)})\) is trivial for the universal blowup family.
\end{lemma}
\begin{proof}
	\sloppy
	Let \(\pi:Z\longrightarrow X\) be the universal blowup family and \(p_x:\text{Bl}_x(X)\longrightarrow X\) be the induced map from the fibre at \(x\) to \(X\). This map corresponds to the projection map for the blowup of \(X\) at \(x\) and is a birational isomorphism, hence it induces an isomorphism between \(H^0(X,\wedge^{2,0}TX)\) and \(H^0(\text{Bl}_x(X_x),\wedge^{2,0}T\text{Bl}_x(X_x))\) \cite[p. 494]{griffiths_principles_1994}. By Serre duality \(H^0(\text{Bl}_x(X_x),\wedge^{2,0}T\text{Bl}_x(X_x))\cong H^2(\text{Bl}_x(X),\mathcal{O}_{\text{Bl}_x(X)})\), hence \(\pi\) composed with Serre duality on the fibres induces an isomorphism between \(H^{2,0}\) and the trivial bundle over \(X\) with fibres \(H^0(X,\wedge^{2,0}TX)\), thus \(H^{2,0}\) is trivial.
\end{proof}
	\begin{proof}[Proof of \Cref{thm:computation of the FSW for the universal blowup family}]
	Recall from \Cref{thm:general computation of the families Seiberg-Witten invariant} that the families Seiberg-Witten invariants are given by
	\[\FSW_n(Z,\pi^*L_1\otimes p^*L_2\otimes \mathcal{O}(kE))=\sum_{m=0}^{h^1-h^2+\rho_g}c_{h^1-h^2+\rho_g-m}(H^{2,0})\Gamma_{m,n}\]
	where \(H^{2,0}\) is the bundle with fibres \(H^2(\text{Bl}_x(X),\mathcal{O}_{\text{Bl}_x(X)})\), from \Cref{lem:H2 is trivial for the blowup bundle} the bundle \(H^{2,0}\) is trivial, so it follows that the only surviving term in the expression for families Seiberg-Witten invariants occurs when \(m=h^1-h^2+\rho_g\), the invariants are then computed as follows
	\begin{equation}
		\label{eqn:blowup family FSW expression}
		\FSW_n(Z,\pi^*L_1\otimes p^*L_2\otimes \mathcal{O}(kE))=\Gamma_{h^1-h^2+\rho_g,n}
	\end{equation}
	where \(h^i=\text{rank}(V^i)=\text{rank}(H^i(X,\pi^*L_1\otimes p^*L_2\otimes \mathcal{O}(kE)))\).
		
	Recall that the vector bundles \(V^i\) have fibres \(\{H^i(Z_x,(\pi^*L_1\otimes p^*L_2\otimes \mathcal{O}(kE)))\}_{x\in X}\), from the projection formula, one has \(V^i=L_1\otimes W^i\) where \(W^i_x=H^i(Z_x,(p^*L_2\otimes \mathcal{O}(kE))_x)\), so it suffices to assume the dimensions of \(\{H^i(Z_x,(p^*L_2\otimes \mathcal{O}(kE))_x)\}_{x\in X}\) are constant over \(X\) for \(i=0,1,2\) to ensure that the \(W^i\) are actually vector bundles.
	
	Furthermore, since \(X\) is a K\"ahler surface, it only has cohomology up to degree 4, hence any terms of higher degree will not survive. These observations combined with standard formulae for the Chern class of the tensor product of a vector bundle with a line bundle and formula for \(\Gamma_{m,n}\) of \Cref{thm:general computation of the families Seiberg-Witten invariant} result in the expressions contained in the statement of the theorem.
\end{proof}
\newpage
	To obtain more explicit results, one must compute the Chern and Segre classes of the vector bundles \(W^i\). We provide a lemma and a subsequent proposition which allows one to do so in certain special cases. Let \(x\in X\) and for simplicity write \(\widetilde{X}=\text{Bl}_x(X)\), let \(I_x\) denote the ideal sheaf of \(x\), so that if \(k\ge 0\) then \(I^k_x\) is then the sheaf of holomorphic functions on \(X\) vanishing at \(x\) to order at least \(k\). Let \(\widetilde{\mathcal{O}}_x\) denote the skyscraper sheaf at \(x\) with values in \(\C\)                                                                                                                          and \(\mathcal{O}_{\widetilde{X}}(-kE)\) be the sheaf of sections of the line bundle \(\mathcal{O}(-kE)\) corresponding to the divisor \(-kE\). We then have the following lemma on the higher direct image sheaves of \(\mathcal{O}_{\widetilde{X}}(-kE)\) under the projection map map \(\).
	\begin{lemma}
		\label{lem:higher direct image sheaves of divisor}
		Let \(k\ge 0\) and \(p:\tilde{X}\rightarrow X\) be the projection map for the blowup of \(X\) at \(x\in X\), then
		\[p_*\mathcal{O}_{\widetilde{X}}(-kE)\cong I^k_x\]
		and
		\[R^jp_*\mathcal{O}_{\widetilde{X}}(-kE)=0\] for \(j>0\)
	\end{lemma}
	\begin{proof}
		Since \(E\) is compact, any holomorphic function on \(E\) is constant, so \(p_*\mathcal{O}_{\widetilde{X}}\subset\mathcal{O}_X\)
		The exceptional divisor \(E\) then gives an exact sequence 
		\[0\rightarrow \mathcal{O}_{\widetilde{X}}(-E)\rightarrow\mathcal{O}_{\widetilde{X}}\rightarrow \mathcal{O}_E\rightarrow 0.\]
		Note that the higher direct image sheaves satisfy \(R^ip_*\mathcal{O}_{\widetilde{X}}=0\) for \(i>0\) \cite[Theorem 9.1]{barth_compact_2004}, thus the induced long exact sequence of higher direct image sheaves gives the following exact sequence
		\[0\rightarrow p_*\mathcal{O}_{\widetilde{X}}(-E)\rightarrow \mathcal{O}_{X}\xrightarrow{\text{ev}_x}\widetilde{\mathcal{O}}_x\rightarrow R^1p_*\mathcal{O}_{\widetilde{X}}(-E)\rightarrow 0\]
		and \(R^jp_*\mathcal{O}_{\widetilde{X}}(-E)=0\) for \(j>1\). Since the evaluation map \(\text{ev}_x:\mathcal{O}_X\rightarrow \widetilde{\mathcal{O}}_x\) is surjective, \(R^1p_*\mathcal{O}_{\widetilde{X}}(-E)=0\), this reduces the above exact sequence. There also a natural map \(I_x\rightarrow \pi_*\mathcal{O}_{\widetilde{X}}(-E)\), this and the structure sequence of \(X\) gives the following commutative diagram with exact rows
		\[\begin{tikzcd}
			0 & {\pi_*\mathcal{O}_{\widetilde{X}}(-E)} & {\mathcal{O}_{X}} & {\widetilde{\mathcal{O}}_x} & 0 \\
			0 & {I_x} & {\mathcal{O}_{X}} & {\widetilde{\mathcal{O}}_x} & 0
			\arrow[from=2-1, to=2-2]
			\arrow[from=1-1, to=1-2]
			\arrow[from=1-2, to=1-3]
			\arrow[from=1-3, to=1-4]
			\arrow[from=1-4, to=1-5]
			\arrow[hook, from=2-2, to=2-3]
			\arrow[from=2-3, to=2-4]
			\arrow[from=2-4, to=2-5]
			\arrow[no head, from=1-4, to=2-4]
			\arrow[from=2-2, to=1-2]
			\arrow[no head, from=1-3, to=2-3]
		\end{tikzcd}\]
		since the rightmost two arrows are isomorphisms, it follows that \(\pi_*\mathcal{O}_{\widetilde{X}}(-E)\cong I_x\).
		
		Now proceed inductively on \(k\), again the exceptional divisor gives an exact sequence
		\[0\rightarrow \mathcal{O}_{\widetilde{X}}(-kE)\rightarrow \mathcal{O}_{\widetilde{X}}(-(k-1)E)\rightarrow \mathcal{O}_E(-(k-1)E)\rightarrow 0\]
		the long exact sequence of higher direct image sheaves and the inductive hypothesis gives the following long exact sequence
		\[0\rightarrow p_*\mathcal{O}_{\widetilde{X}}(-kE)\rightarrow p_*\mathcal{O}_{\widetilde{X}}(-(k-1)E)\rightarrow p_*\mathcal{O}_{E}(-(k-1)E)\rightarrow R^1\pi_*\mathcal{O}_{\widetilde{X}}(-kE)\rightarrow 0.\]
		and that
		\[R^jp_*(-(k-1)E)\cong R^{j+1}p_*\mathcal{O}_{\widetilde{X}}(-kE)\]
		for \(j>0.\) Since \(\mathcal{O}(E)|_E\cong \mathcal{O}(-1)\) over \(\C\P^1\), it follows that \(\mathcal{O}(-(k-1)E)\cong \widetilde{\mathcal{O}}_x(S^k(T^*_xX))\) and \(R^jp_*\mathcal{O}_E(-(k-1)E)=0\) for \(j>0\).
		
		Consequently \(R^jp_*\mathcal{O}_{\widetilde{X}}(-kE)=0\) for \(j>1\) and the above exact sequence is
		\[0\rightarrow p_*\mathcal{O}_{\widetilde{X}}(-kE)\rightarrow p_*\mathcal{O}_{\widetilde{X}}(-(k-1)E)\rightarrow \widetilde{\mathcal{O}}_x(S^kT^*_xX)\rightarrow R^1p_*\mathcal{O}_{\widetilde{X}}(-kE)\rightarrow 0\]
		The map \(p_*\mathcal{O}_{\widetilde{X}}(-(k-1)E)\rightarrow \widetilde{\mathcal{O}}_x(S^kT^*_xX)\) is surjective since given a choice of local coordinates around \(x\), an element \(f\in S^k(T^*_xX)\) defines a local holomorphic function vanishing to order \(k-1\) at \(x\), thus defining a local section of \(p_*\mathcal{O}_{\widetilde{X}}(-(k-1)E)\) and this evaluates to \(f\) at \(x\). Hence \(R^1p_*\mathcal{O}_{\widetilde{X}}(-kE)=0\) and we have proven the second claim. This also reduces the exact sequence to
		\[0\longrightarrow p_*\mathcal{O}_{\widetilde{X}}(-kE)\rightarrow p_*\mathcal{O}_{\widetilde{X}}(-(k-1)E)\rightarrow \widetilde{\mathcal{O}}_x(S^kT^*_xX)\rightarrow 0.\]	
		A section of \(I^k_x\) can be regarded as a map which vanishes at \(x\) to order at least \(k\), there are natural maps \(I^k_x\longrightarrow p_*\mathcal{O}_{\widetilde{X}}(-kE)\) making the following diagram commute
		\[\begin{tikzcd}
			0 & {I^k_x} & {I^{k-1}_x} & {\widetilde{\mathcal{O}}_x(S^kT^*_xX)} & 0 \\
			0 & {p_*\mathcal{O}_{\widetilde{X}}(-kE)} & {p_*\mathcal{O}_{\widetilde{X}}(-(k-1)E)} & {\widetilde{\mathcal{O}}_x(S^kT^*_xX)} & 0
			\arrow[from=1-1, to=1-2]
			\arrow[from=1-2, to=1-3]
			\arrow[from=1-3, to=1-4]
			\arrow[from=1-4, to=1-5]
			\arrow[from=2-1, to=2-2]
			\arrow[from=2-2, to=2-3]
			\arrow[from=2-3, to=2-4]
			\arrow[from=2-4, to=2-5]
			\arrow[from=1-2, to=2-2]
			\arrow[from=1-4, to=2-4]
			\arrow[from=1-3, to=2-3]
		\end{tikzcd}\]
		where the rows are exact. The rightmost map is the identity and the middle is an isomorphism by the inductive hypothesis. Consequently, the leftmost vertical map is an isomorphism and
		\[p_*\mathcal{O}_{\widetilde{X}}(-kE)\cong I^k_x\]
		so we obtain the first claim by induction.
	\end{proof}
	\begin{theorem}
		\label{thm:exact sequences for blowups}
		Let \(L\) be a holomorphic line bundle on \(X\). For each \(k\ge 0\), the following holds.
		\begin{enumerate}[(1)]
			\item There is a long exact sequence
			\begin{equation*}
				\label{eqn:blowup exact sequence 1}
				\begin{tikzcd}
					0 & {H^0(\widetilde{X},p^*L\otimes \mathcal{O}(-kE))} & {H^0(X,L)} & {L_x\otimes(\mathcal{O}_X/I^k_x)} \\
					& {H^1(\widetilde{X},p^*L\otimes \mathcal{O}(-kE))} & {H^1(X,L)} & 0 & {}
					\arrow[from=1-1, to=1-2]
					\arrow[from=1-2, to=1-3]
					\arrow["{\text{ev}_x}", from=1-3, to=1-4]
					\arrow[from=1-4, to=2-2]
					\arrow[from=2-2, to=2-3]
					\arrow[from=2-3, to=2-4]
				\end{tikzcd}
			\end{equation*}
			\item There is an isomorphism \(H^2(\widetilde{X},p^*L\otimes \mathcal{O}(-kE))\cong H^2(X,L)\)
			\item For \(m\ge 1\), there is a long exact sequence
			\begin{equation*}
				\begin{tikzcd}
					\label{eqn:blowup exact sequence 2}
					0 & {\rightarrow H^1(X,L)} & {H^1(\widetilde{X},p^*L\otimes\mathcal{O}(kE))} & {L_x\otimes (K_X^*)_x\otimes (\mathcal{O}_X/I^k_x)^*} \\
					& {H^2(X,L)} & {H^2(\widetilde{X},p^*L\otimes \mathcal{O}(kE))} & 0 & {}
					\arrow[from=1-1, to=1-2]
					\arrow[from=1-2, to=1-3]
					\arrow[from=1-3, to=1-4]
					\arrow[from=1-4, to=2-2]
					\arrow[from=2-2, to=2-3]
					\arrow[from=2-3, to=2-4]
				\end{tikzcd}
			\end{equation*}
			\item and an isomorphism (for \(m\ge 1\))
			\[H^0(\widetilde{X},p^*L\otimes \mathcal{O}(kE))\cong H^0(X,L)\]
		\end{enumerate}
	\end{theorem}
	\begin{proof}
	Since \(L\) is locally trivial, it corresponds to a locally free sheaf, hence via the projection formula, it follows that
	\[R^qp_*\mathcal{O}_{\widetilde{X}}(-kE)\otimes \mathcal{O}_X(L)\cong R^q p_*(\mathcal{O}_{\widetilde{X}}(-kE)\otimes \mathcal{O}_{\widetilde{X}}(L))=R^q\pi_*\mathcal{O}_{\widetilde{X}}(p^*L\otimes\mathcal{O}(-kE))\]
	Consequently, from \Cref{lem:higher direct image sheaves of divisor} we have that
	\[R^qp_*\mathcal{O}_{\widetilde{X}}(p^*L\otimes\mathcal{O}(-kE))=0\]
	for \(q>0\), and 
	\[p_*\mathcal{O}_{\widetilde{X}}(p^*L\otimes\mathcal{O}(-kE))=\mathcal{O}_X(L)\otimes_{\mathcal{O}_X}I^k_x\]
	in the case when \(q=0\).
	
	The \(E_2\) page of the Leray spectral sequence for \(\mathcal{O}_{\widetilde{X}}(p^*L\otimes \mathcal{O}(-kE))\) under the map \(p:\widetilde{X}\longrightarrow X\) is given by \(E_2^{m,n}=H^m(X,R^np_*\mathcal{O}_{\widetilde{X}}(p^*L\otimes\mathcal{O}(-kE)))\), this converges to \(H^{m+n}(\widetilde{X},\mathcal{O}_{\widetilde{X}}(p^*L\otimes\mathcal{O}(-kE)))\equiv H^{m+n}(\widetilde{X},p^*L\otimes \mathcal{O}(-kE))\). However, for any \(m\), \(E_2^{m,n}\) is zero for \(n>0\) by the previous computation, hence the spectral sequence degenerates at the \(E_2\) level and so	
	\[H^j(\widetilde{X},p^*L\otimes\mathcal{O}(-kE))\cong H^j(X,p_*\mathcal{O}_{\widetilde{X}}(p^*L\otimes\mathcal{O}(-kE)))\cong H^j(X,\mathcal{O}_X(L)\otimes_{\mathcal{O}_X}I^k_x).\]
	
	Since \(\mathcal{O}_X(L)\) is a locally free sheaf, by tensoring the exact sequence defining the quotient \(\mathcal{O}_X/I^k_x\) the following sequence is exact.
	\[0\rightarrow \mathcal{O}_X(L)\otimes_{\mathcal{O}_X}I^k_x\rightarrow \mathcal{O}_X(L)\rightarrow L_x\otimes (\mathcal{O}_X/I^k_x)\rightarrow 0.\]
	Taking the induced long exact sequence in cohomology and using the fact that \(\mathcal{O}_X/I^k_x\) has no cohomology in degrees bigger than zero since it is a skyscraper sheaf, \((1)\) and \((2)\) immediately follow. If we then consider \((1)\) and \((2)\) for line bundles over \(\widetilde{X}\) of the form \(p^*L\otimes p^*K_X^*\otimes\mathcal{O}(-(k-1)E)\) when \(k\ge 1\), then \((3.)\) and \((4.)\) follow from applying Serre duality, the fact that \(K_{\widetilde{X}}=p^*K_X\otimes\mathcal{O}(E)\) and dualising the resulting exact sequence and isomorphism respectively.
	\end{proof}
	Since there are two different sets of exact sequences and isomorphisms pertaining to the cohomology of line bundles over the blowup of \(X\), we hereafter set \(k\ge 0\) and write a line bundle on \(\tilde{X}=\text{Bl}_x(X)\) as \(p^*L_2\otimes \mathcal{O}(\pm kE)\).
	
	\subsection{Line bundles of the form \(\mathbf{\pi^*L_1\otimes p^*L_2\otimes \mathcal{O}(-kE)}\)}
	\label{sec:blowup family with k<0}
	In this section we consider the families Seiberg-Witten invariants for \(\text{spin}^c\) structures determined by a line bundle on the universal blowup family of the form \(\pi^*L_1\otimes p^*L_2\otimes \mathcal{O}(-kE)\) where \(k\ge 0\). Recall from \Cref{thm:exact sequences for blowups}, there is a fibrewise isomorphism
	\[H^2(Z_x,p^*L_2\otimes\mathcal{O}(-kE))\cong H^2(X,L_2)\]
	hence it defines a trivial vector bundle \(W^2=H^2(Z,p^*L_2\otimes\mathcal{O}(-kE))\) with fibres \(H^2(X,L_2)\). We also have the following exact sequences of vector spaces on each fibre
	\begin{equation}
		\label{eqn:families blowup exact sequence}
		\begin{tikzcd}
			0 & {H^0(Z_x,p^*L_2\otimes\mathcal{O}(-kE))} & {H^0(X,L_2)} & {L_{2,x}\otimes(\frac{\mathcal{O}_X}{I^k_x})} \\
			& {H^1(Z_x,p^*L_2\otimes\mathcal{O}(-kE))} & {H^1(X,L_2)} & 0
			\arrow[from=1-1, to=1-2]
			\arrow[from=1-2, to=1-3]
			\arrow["{\text{ev}_x}", from=1-3, to=1-4]
			\arrow[from=1-4, to=2-2]
			\arrow[from=2-2, to=2-3]
			\arrow[from=2-3, to=2-4]
		\end{tikzcd}.
	\end{equation}
	There is an isomorphism \(\mathcal{O}_X/I^k_x\cong \widetilde{\mathcal{O}}_x(S^{\le (k-1)}(T_x^*X))\) where \(\widetilde{\mathcal{O}}_x(S^{\le (k-1)}(T_x^*X))\) is the skyscraper sheaf at \(x\) with values in \(S^{\le (k-1)}(T_x^*X)\) and \(S^{\le (k-1)}(T_x^*X)\) is the \(k-1\)-th symmetric power of \(T_x^*X\). Since given a vector space \(V\), the space \(S^k(V^*)\) can be identified with the degree \(k\) polynomials on \(V\) and we have an isomorphism given by a local chart containing \(p\), \(J^k_x\cong \C[z_1,z_2]/(z_1-p_1,z_2-p_2)^{k+1}\) where \(J^k\) denotes the \(k\)-th jet bundle of \(X\). It follows that  \(\mathcal{O}_X/I^k_x\cong J^{k-1}_x\) and so \(L_{2,x}\otimes \mathcal{O}_X/I^k_x=J^{k-1}_x(L_2)\), where \(J^{k-1}(L_2)\) is the \(k-1\)-th jet bundle of \(L_2\).
	
	Assume the dimensions of the cohomology groups \(\{H^i(Z_x,p^*L_2\otimes\mathcal{O}(-kE))\}_{x\in X}\) are constant so that they define vector bundles \(W^i\), a priori we have not established this, although we shall only be concerned with computations in cases where this can be established from the exact sequence of vector spaces. Proceeding with this assumption, there is an exact sequence of vector bundles over \(X\)
	\[\begin{tikzcd}
		0 & W^0 & {H^0(X,L_2)} & {J^{k-1}(L_2)} \\
		& W^1 & {H^1(X,L_2)} & 0
		\arrow[from=1-1, to=1-2]
		\arrow[from=1-2, to=1-3]
		\arrow[from=1-3, to=1-4]
		\arrow[from=1-4, to=2-2]
		\arrow[from=2-2, to=2-3]
		\arrow[from=2-3, to=2-4]
	\end{tikzcd}.\]
	Note that the rank of \(J^{k-1}(L_2)\) is easily shown to be \(\text{dim}(S^{\le (k-1)}(T_x^*X))=\sum_{j=0}^{k-1}(j+1)=k(k+1)/2\) and the ranks of \(W^i\) and \(V^i\) are identical. It then follows from the above exact sequence that if \(h^i=\text{rank}(W^i)\) and \(p^i=\text{dim}(H^i(X,L_2))\)
	\[h^0-p^0+k(k+1)/2-h^1+p^1=0.\]
	Since the fibrewise dimension of the families Seiberg-Witten moduli space is \(h^0-1\), the families Seiberg-Witten invariants are necessarily zero unless
	\[h^0=p^0+h^1-p^1-k(k+1)/2>0.\]
	We now shall consider some special cases in further detail.
	
	\begin{example} \label{ex:k=0 case}
	An easy computation can be obtained in the case when \(k=0\). When \(k=0\) \Cref{thm:exact sequences for blowups} implies that there are fibrewise isomorphisms \(W^i\cong H^i(X,L_2)\) for \(i=0,1\). It follows that the \(W^i\) are all trivial vector bundles with rank \(h^i=p^i(L_2)\), thus \(V^i\cong L_1^{\oplus p^i}\) and so the invariant is zero unless \(p^0>0\), another immediate consequence is that \(\delta=\rho_g+1-\chi(X,L_2)+n\).
	If \(\delta=0\) then the families invariants coincides with the ordinary Seiberg-Witten invariant of \(X\). When \(\delta=1\), since the \(W^i\) are trivial, that the invariant is given by
	\begin{equation}
		\label{eqn:FSW blowup for trivial W and delta=1 triv}
		\begin{split}
			\FSW(Z,L_1,L_2)=(-1)^{n+1}p^0c_1(L_1)
		\end{split}
	\end{equation}
when \(p^1-p^2+\rho_g=0\), and
	\begin{equation}
		\label{eqn:FSW blowup for trivial W and delta=1}
		\begin{split}
			\FSW(Z,L_1,L_2)&=(-1)^{n}(p^1-p^2)c_1(L_1)\binom{p^1-p^2-1}{p^1-p^2+\rho_g-1} \\
			&+(-1)^{n+1}p^0c_1(L_1)\binom{p^1-p^2}{p^1-p^2+\rho_g}
		\end{split}
	\end{equation}
when \(p^1-p^2+\rho_g>0\), where we have expanded terms for the Chern and Segre classes of \(W^i\otimes L\) in \Cref{thm:computation of the FSW for the universal blowup family}. For \(\delta=2\) the invariant is given by the same expressions as in \Cref{thm:computation of the FSW for the universal blowup family} with \(h^i=p^i\), this could be made further explicit by expanding the tensor product terms for the Chern and Segre classes, although we avoid doing so for simplicity.
\end{example}
	
\label{ex:k=1 case}
\begin{example} 
	It is instructive to consider the computation for the invariants when for a line bundle on the universal blowup family of the form \(\pi^*L_1\otimes p^*L_2\otimes \mathcal{O}(-E)\), since it indicates the assumptions required to be able to compute the Chern and Segre classes of the vector bundles \(W^i\) more generally. Because of our choice of \(k\), it follows that \(\mathcal{O}_X/I_x\cong \mathcal{O}_x(\C)\), consequently by \Cref{thm:exact sequences for blowups}, the vector bundle \(W^2\) is trivial with fibres \(H^2(X,L_2)\) and we obtain the following exact sequence of vector spaces
	\[
	\begin{tikzcd}
		0 & {H^0(Z_x,p^*L_2\otimes\mathcal{O}(-kE))} & {H^0(X,L_2)} & {L_{2,x}} \\
		& {H^1(Z_x,p^*L_2\otimes\mathcal{O}(-kE))} & {H^1(X,L_2)} & 0
		\arrow[from=1-1, to=1-2]
		\arrow[from=1-2, to=1-3]
		\arrow[from=1-3, to=1-4]
		\arrow[from=1-4, to=2-2]
		\arrow[from=2-2, to=2-3]
		\arrow[from=2-3, to=2-4]
	\end{tikzcd}
	\]
	We shall make a further assumption so that an explicit computation can be made, suppose that the line bundle \(L_2\) is \textit{basepoint-free}, that is, for each \(x\in X\) there exists a non-zero holomporhic section of \(L_2\) with \(s(x)\neq 0\), then it necessarily follows that the map \(H^0(X,L_2)\longrightarrow L_{2,x}\) is surjective for all \(x\) and dimensions of the cohomology groups are then constant along \(b\). Consequently, by exactness \(H^1(Z_b,L_2\otimes \mathcal{O}(-E))\cong H^1(X,L_2)\) so \(W^1\) is also trivial and the following sequence of vector bundles over \(X\) is exact
	\[0\rightarrow W^0\rightarrow H^0(X,L_2)\rightarrow L_2\rightarrow 0.\]
	It follows that \(h^0=p^0-1\), hence the invariant is zero unless \(p^0=\text{dim}(H^0(X,L_2))>1\), since \(H^0(X,L_2)\) is a trivial vector bundle, we obtain an expression from the above exact sequence in the total Chern classes, namely \(c(L_2)c(W^0)=1\). Hence \(s(W^0)=c(L_2)\), it follows that \(s_1(W^0\otimes L_1)=c_1(L_2)+
	(1-p^0)c_1(L_1)\) and \(s_2(W^0\otimes L_1)=c_1(L_1)^2p^0(p^0-1)/2-p^0c_1(L_1)c_1(L_2)\). One may then use these expressions in conjunction with the formulae of \Cref{thm:computation of the FSW for the universal blowup family} to obtain expressions for the families Seiberg-Witten invariants.
	
	The expression of the invariant when \(\delta=1\) is then following
	\begin{equation}
		\label{eqn:FSW for blowup when k=1 and delta=1}
		\begin{split}
			\FSW(Z,L_1,L_2,-1)=(-1)^{n}(c_1(L_2)-(p^0-1)c_1(L_1))
		\end{split}
	\end{equation}
when \(p^1-p^2+\rho_g=0\), and is
	\begin{equation}
		\label{eqn:FSW for blowup when k=1 and delta=1 rho 0}
		\begin{split}
			\FSW(Z,L_1,L_2,-1)&=(-1)^{n}p^1c_1(L_1)\binom{p^1-p^2-1}{p^1-p^2+\rho_g-1} \\
			&+(-1)^{n}(c_1(L_2)-(p^0-1)c_1(L_1))\binom{p^1-p^2}{p^1-p^2+\rho_g} \\
			&+(-1)^{n+1}p^2c_1(L_1)\binom{p^1-p^2-1}{p^1-p^2+\rho_g-1}
		\end{split}
	\end{equation}
for \(p^1-p^2+\rho_g>0\). When \(\delta=2\) we have
\begin{equation}
	\label{eqn:FSW for blowup when k=1 and delta=2 rho pos}
	\begin{split}
		\FSW(Z,L_1,L_2,-1)=(-1)^{n}\left(\frac{p^0(p^0-1)}{2}c_1(L_1)^2-p^0c_1(L_1)c_1(L_2)\right)
	\end{split}
\end{equation}
when \(p^1-p^2+\rho_g=0\),
\begin{equation}
	\label{eqn:FSW for blowup when k=1 and delta=2}
	\begin{split}
		\FSW(Z,L_1,L_2,-1)&=(-1)^{n+1}p^1(c_1(L_1)c_1(L_2)+(1-p^0)c_1(L_1)^2) \\
		&+(-1)^{n}\left(\frac{p^0(p^0-1)}{2}c_1(L_1)^2-p^0c_1(L_1)c_1(L_2)\right)(p^1-p^2) \\
		&+(-1)^{n+1}p^2(c_1(L_1)c_1(L_2)+(1-p^0)c_1(L_1)^2)
	\end{split}
\end{equation}
when \(p^1-p^2+\rho_g=1\) and
	\begin{equation}
		\label{eqn:FSW for blowup when k=1 and delta=2 rho 1}
		\begin{split}
			\FSW(Z,L_1,L_2,-1)&=(-1)^{n}\frac{p^1(p^1-1)}{2}c_1(L_1)^2\binom{p^1-p^2-2}{p^1-p^2+\rho_g-2} \\
			&+(-1)^{n+1}p^1(c_1(L_1)c_1(L_2)+(1-p^0)c_1(L_1)^2)\binom{p^1-p^2-1}{p^1-p^2+\rho_g-1} \\
			&+(-1)^{n+1}p^2p^1c_1(L_1)^2\binom{p^1-p^2-2}{p^1-p^2+\rho_g-2} \\
			&+(-1)^{n}\left(\frac{p^0(p^0-1)}{2}c_1(L_1)^2-p^0c_1(L_1)c_1(L_2)\right)\binom{p^1-p^2}{p^1-p^2+\rho_g} \\
			&+(-1)^{n+1}p^2(c_1(L_1)c_1(L_2)+(1-p^0)c_1(L_1)^2)\binom{p^1-p^2-1}{p^1-p^2+\rho_g-1} \\
			&+(-1)^{n}\frac{p^2(p^2-1)}{2}c_1(L_1)^2\binom{p^1-p^2-2}{p^1-p^2+\rho_g-2}.
		\end{split}
	\end{equation}
for \(p^1-p^2+\rho_g>1\).
\end{example}

	\begin{remark}\label{rem:general k}
	Via similar assumptions to the case when \(k=1\), one can in principle, make computations for the families Seiberg-Witten invariants of the universal blowup family for a line bundle of the form \(\pi^*L_1\otimes p^*L_2\otimes \mathcal{O}(-kE)\). Recall the exact sequence of vector spaces
	\[\begin{tikzcd}
		0 & {H^0(Z_b,p^*L_2\otimes\mathcal{O}(-kE))} & {H^0(X,L_2)} & {L_{2,x}\otimes(\frac{\mathcal{O}_X}{I^k_x})} \\
		& {H^1(Z_b,p^*L_2\otimes\mathcal{O}(-kE))} & {H^1(X,L_2)} & 0
		\arrow[from=1-1, to=1-2]
		\arrow[from=1-2, to=1-3]
		\arrow["{\text{ev}_x}", from=1-3, to=1-4]
		\arrow[from=1-4, to=2-2]
		\arrow[from=2-2, to=2-3]
		\arrow[from=2-3, to=2-4]
	\end{tikzcd}\]
	and the isomorphism \(\mathcal{O}_X/I^k_x\cong \mathcal{O}_x(S^{\le (k-1)}(T^*_xX))\). 
	
	If we assume as in the \(k=1\) case that the map in the long exact sequence induced by the evaluation map is surjective for each \(x\in X\). Then by identical reasoning one obtains that \(W^1\) is trivial and there is an exact sequence of vector bundles
	\[0\rightarrow H^0(Z,p^*L_2\otimes\mathcal{O}(-E))\rightarrow H^0(X,L_2)\rightarrow J^{k-1}(L_2)\rightarrow 0.\]
	it follows that \(s(W^0)=c(J^{k-1}(L))\). There is also an exact sequence
	\[0\rightarrow S^k(T^*X)\rightarrow J^k\rightarrow J^{k-1}\rightarrow 0\]
	with the map on the right being the obvious projection map sending a \(k\)-jet to its corresponding \(k-1\) jet and the map on the left is the inclusion of degree \(k\)-polynomials. Tensoring by \(L_2\) induces the following exact sequence
	\[0\rightarrow S^q(T^*X)\otimes L_2\rightarrow J^q(L_2)\rightarrow J^{q-1}(L_2)\rightarrow 0\]
	from which the total Chern class of the \(q\)-th jet bundle can be inductively computed as
	\[c(J^q(L_2))=c(S^q(T^*X)\otimes L_2)c(J^{q-1}(L_2))\]
	consequently, one may compute the Segre classes \(s_j(W^0\otimes L_2)\) and use the formulae of \Cref{thm:computation of the FSW for the universal blowup family} to obtain an expression for the invariants.
\end{remark}
	
	\subsection{Line bundles of the form \(\mathbf{\pi^*L_1\otimes p^*L_2\otimes \mathcal{O}(kE)}\)\label{sec:blowup family with k>0}}
	 We briefly discuss some means to compute the Chern and Segre classes of the line bundles \(W^i\) for the universal blowup family for a line bundle of the form \(\pi^*L_1\otimes p^*L_2\otimes \mathcal{O}(kE)\), although this proceeds similarly to \Cref{sec:blowup family with k<0} since the exact sequences involved are simply obtained by applying Serre duality to the first exact sequence of \Cref{thm:exact sequences for blowups} for a line bundle of the form \(p^*L_2^*\otimes p^*K_X\otimes\mathcal{O}(-(k-1)E)\) and dualising the resulting exact sequence.
	 Recall that there is an isomorphism \(W^0\cong H^0(X,L_2)\) and an exact sequence
	 \[
	 \begin{tikzcd}
	 	0 & {\rightarrow H^1(X,L_2)} & {H^1(Z_x,p^*L_2\otimes\mathcal{O}(kE))} & {L_x\otimes (K_X^*)_x\otimes (\mathcal{O}_X/I^k_x)^*} \\
	 	& {H^2(X,L_2)} & {H^2(Z_x,p^*L_2\otimes \mathcal{O}(kE))} & 0 & {}
	 	\arrow[from=1-1, to=1-2]
	 	\arrow[from=1-2, to=1-3]
	 	\arrow[from=1-3, to=1-4]
	 	\arrow[from=1-4, to=2-2]
	 	\arrow[from=2-2, to=2-3]
	 	\arrow[from=2-3, to=2-4]
	 \end{tikzcd}
	 \]
	 hence \(W^0\) is trivial and the invariants are zero unless \(\text{dim}(H^0(X,L_2))>0\). We may then obtain computations in special cases as before.
	 \begin{example}
	 If \(k=1\), then we are in an identical situation as the \(k=0\) case of \Cref{sec:blowup family with k<0}. That is, we have that \(W^1\cong H^1(X,L_2)\) and \(W^2\cong H^2(X,L_2)\), thus the vector bundles \(W^i\) are all trivial so one obtains an identical computation of the invariants as before in \Cref{ex:k=0 case}.
	 \end{example}
 	More generally, provided, the dimensions of \(H^i(Z_x,p^*L_2\otimes\mathcal{O}(kE))\) are constant, we obtain an exact sequence of vector bundles
 \[0\rightarrow H^1(X,L_2)\rightarrow W^1\rightarrow J^{k-2}(L_2\otimes K_X^*)^*\rightarrow H^2(X,L_2)\rightarrow W^2\rightarrow 0\]
 Recall from the proof of \Cref{thm:exact sequences for blowups}, that the fibrewise version of this exact sequence is obtained by applying Serre duality and dualising the first exact sequence of \Cref{thm:exact sequences for blowups} for a line bundle of the form \(p^*L_2^*\otimes p^*K_X\otimes\mathcal{O}(-(k-1)E)\). Consequently, we may apply similar assumptions to those in the previous section to the line bundle \(L_2\otimes K_X^*\) to obtain a reduction of the exact sequence. 
 \begin{example}
 	As an example, if \(L_2^*\otimes K_X\) is basepoint free it follows that \(W^2\cong H^2(X,L_2)\) and we obtain that \(H^1(Z_x,p^*L_2\otimes\mathcal{O}(kE))\) has constant dimension from the fibrewise exact sequence. Thus there is an exact sequence of vector bundles for \(k=2\)
 	\[0\rightarrow L_2\otimes K_X^*\rightarrow  W^2\rightarrow H^2(X,L_2) \rightarrow 0\]
 	thus \(c_i(L_2\otimes K_X^*)=c_i(W^2)\) and one obtains formulae for the families Seiberg-Witten invariants similar to \Cref{ex:k=1 case} with \(L_2\) replaced by \(L_2\otimes K_X^*\) and the \(p^0,p^2\) factors not inside the binomial coefficients interchanged.
 	
 	 Similar to \Cref{sec:blowup family with k<0}, if we assume more generally that the map \(H^0(X,L_2^*\otimes K_X)\rightarrow (L_2^*\otimes K_X)_x\otimes \frac{\mathcal{O}_X}{I^k_x}\) is surjective for each \(x\), then by applying Serre duality and dualising the resulting exact sequence, one obtains that the dimensions of the relevant cohomology groups are constant, thus defining vector bundles \(W^i\) where \(W^1\) is trivial, we also obtain the following exact sequence of vector bundles
 	 \[0\rightarrow J^{k-2}(L_2\otimes K_X^*)\rightarrow  W^2\rightarrow H^2(X,L_2) \rightarrow 0.\]
 	 One may use this to compute the Chern and Segre classes of \(W^2\) in terms of those of \(J^{k-2}(L_2^*\otimes K_X)\), thus the families Seiberg-Witten invariants may then be computed in principle as discussed in \Cref{rem:general k}.
 \end{example}
	
	
	
	
	\printbibliography
\end{document}